\documentclass[sn-mathphys-num]{sn-jnl}

\usepackage{graphicx}%
\usepackage{caption}
\usepackage{float}%
\usepackage{multirow}%
\usepackage{amsmath,amssymb,amsfonts}%
\usepackage{amsthm}%
\usepackage{mathrsfs}%
\usepackage[title]{appendix}%
\usepackage{xcolor}%
\usepackage{textcomp}%
\usepackage{manyfoot}%
\usepackage{booktabs}%
\usepackage{algorithm}%
\usepackage{algorithmicx}%
\usepackage{algpseudocode}%
\usepackage{listings}%
\usepackage[short,nocomma]{optidef} 
\usepackage{cleveref}
\usepackage{tcolorbox}
\usepackage{tikz}%
\usepackage{subcaption}%
\usepackage{graphicx}%
\usepackage{soul}
\usepackage{tabularray}
\usepackage{needspace}%
\usepackage{adjustbox}
\usepackage{caption}
\usepackage{hyperref}
\usetikzlibrary{quantikz2}%

\theoremstyle{thmstyleone}%
\newtheorem{theorem}{Theorem}%
\newtheorem{proposition}[theorem]{Proposition}%

\theoremstyle{thmstyletwo}%

\theoremstyle{thmstylethree}%

\newcommand\blfootnote[1]{%
  \begingroup
  \renewcommand\thefootnote{}\footnote{#1}%
  \addtocounter{footnote}{-1}%
  \endgroup
}

\begin{document}

\title[Article Title]{Optimizing Multiple-Control Toffoli Quantum Circuit Design with Constraint Programming\footnote{This paper is an extended version of the conference paper:
\textit{Jihye Jung, Kevin Dalmeijer, and Pascal Van Hentenryck. A New Optimization Model for Multiple-Control Toffoli Quantum Circuit Design. In 30th International Conference on Principles and Practice of Constraint Programming (CP 2024). Leibniz International Proceedings in Informatics (LIPIcs), Volume 307, pp. 16:1-16:20, Schloss Dagstuhl – Leibniz-Zentrum für Informatik (2024),} \url{https://doi.org/10.4230/LIPIcs.CP.2024.16}.}}

\author*[1]{\fnm{Jihye} \sur{Jung}}\email{jihye.jung@gatech.edu}
\author[1]{\fnm{Kevin} \sur{Dalmeijer}}\email{dalmeijer@gatech.edu}
\author[1]{\fnm{Pascal} \sur{Van Hentenryck}}\email{pvh@gatech.edu}

\affil*[1]{\orgdiv{H. Milton Stewart School of Industrial and Systems Engineering}, \orgname{Georgia Institute of Technology}, \orgaddress{\street{North Avenue}, \city{Atlanta}, \postcode{30332}, \state{Georgia}, \country{United States}}}

\abstract{As quantum technology advances, the efficient design of quantum circuits has become an important area of research.
This paper provides an introduction to the MCT quantum circuit design problem for reversible Boolean functions with the necessary background in quantum computing to comprehend the problem.
While this is a well-studied problem, optimization models that minimize the true objective have only been explored recently.
This paper introduces a new optimization model and symmetry-breaking constraints that improve solving time by up to two orders of magnitude compared to earlier work when a Constraint Programming solver is used.
Experiments with up to seven qubits and using up to 15 quantum gates result in several new best-known circuits, obtained by any method, for well-known benchmarks.
Several in-depth analyses are presented to validate the effectiveness of the symmetry-breaking constraints from multiple perspectives.
Finally, an extensive comparison with other approaches shows that optimization models may require more time but can provide superior circuits with optimality guarantees.}


\maketitle
\blfootnote{ORCID: (Jihye Jung) 0000-0002-5217-020X, (Kevin Dalmeijer) 0000-0002-4304-7517, (Pascal Van Hentenryck) 0000-0001-7085-9994}

\section{Introduction}\label{sec:introduction}

Quantum computing has gained significant attention due to its potential for achieving computational supremacy, as demonstrated by well-known quantum algorithms such as Shor’s algorithm for prime factorization \cite{shor1999polynomial} and Grover’s algorithm for unstructured search \cite{grover1996fast}.
With the rapid advancements in quantum technologies in recent years, the efficient design of quantum circuits has emerged as a crucial area of research. A fundamental challenge in quantum circuit design is to implement a \emph{target function}, using gates selected from a \emph{preset gate library}, while minimizing \emph{circuit costs} according to a given metric. Different choices for these three components lead to multiple versions of the target problem.
This work considers a target problem specified as follows:

\begin{itemize}
\item[] \textbf{Target function.} \emph{Reversible Boolean function}, a key component at the core of most quantum algorithms.
\item[] \textbf{Preset gate library.} \emph{Multiple-Control Toffoli (MCT) gate}, a typical high-level gate commonly used to represent reversible Boolean functions.
\item[] \textbf{Circuit cost.} \emph{Quantum cost}, the number of low-level quantum gates required to realize the high-level gates in the circuit.
\end{itemize}

These concepts are introduced in detail in Section~\ref{sec:terminology}, with an appropriate background in quantum computing provided to the reader.
The primary goal of this paper is to introduce a new optimization model to design quantum circuits within the setting defined above.
For brevity, this problem will be referred to as the \emph{MCT quantum circuit design problem}.

\subsection{Literature Review}
For similar target problems, early methods for quantum circuit design were developed based on intuitive observations and preconfigured circuit templates.
Studies in this stage constructed a base library of small-scaled circuits to heuristically synthesize larger circuits for reversible Boolean functions \cite{maslov2005toffoli, golubitsky2011study}. 
Post-synthesis algorithms, such as relocation algorithms \cite{abdessaied2013exact, prasad2006data, maslov2007techniques, maslov2008quantum}, were introduced to further improve these circuits, although the improved results cannot guarantee optimality.

Several papers have used different representations of reversible Boolean functions to develop efficient synthesis algorithms.
Cycle representation was used to devise several decomposition-based approaches \cite{saeedi2010reversible, zhu2018reversible}. 
In particular, reference \cite{saeedi2010reversible} reports an average cost improvement of 20\% for benchmark functions with up to 20 qubits. Other approaches have leveraged the Reed-Muller expansion to decompose reversible Boolean functions into exclusive-OR terms of Boolean products.
The Reed-Muller expansion and the corresponding decision diagrams have appeared in \cite{gupta2006algorithm} and \cite{lin2014rmdds} to address functions with up to 30 and 15 qubits, respectively.
Reference \cite{lin2014rmdds} demonstrates a cost improvement of approximately 35\% compared to previous studies, using a time limit of 600 seconds. 
A comparative analysis of decision diagram approaches for Reed-Muller expansion is also proposed \cite{wille2010effect}.

Another heuristic approach uses a quantum multiple-valued decision diagram, an efficient representation for matrices, to handle both reversible and irreversible functions \cite{zulehner2017skipping}.
The authors demonstrate the high scalability of the algorithm, handling functions with states up to length 156. 
An A* algorithm was applied to the problem with an approximate heuristic function deduced from observations on state transitions \cite{datta2012synthesis}, while others used heuristics based on isomorphic subgraph matching \cite{krishna2014efficient} and window optimization \cite{soeken2010window}. 
Evolutionary algorithms such as genetic algorithms \cite{alfailakawi2015depth}, adaptive genetic algorithms \cite{sasamal2015reversible}, genetic programming \cite{abubakar2017reversible}, tabu search \cite{de2018reversible}, and particle swarm optimization \cite{datta2012particle} were also suggested to obtain near-optimal solutions. 
While these methods offer different trade-offs between computation time and circuit quality, they are all heuristic and do not provide optimality guarantees.

Meanwhile, exact synthesis methodologies have been proposed to obtain optimal quantum circuits. 
Reference \cite{grosse2009exact} iteratively solves satisfiability problems to obtain a quantum circuit with the minimum number of gates.
This exact approach handles benchmark functions with three up to six qubits within a maximum of 5,000 seconds of computing time.
A method based on quantified Boolean formula satisfiability, a generalized version of Boolean satisfiability, is also proposed to handle the same problem \cite{wille2008quantified}.
The authors report results for functions with four up to six qubits using a 2,000-second time limit.
The exact methods notably handle relatively small functions but find better solutions within this space.
Both references \cite{grosse2009exact} and \cite{wille2008quantified} report results in terms of quantum cost, which is the total number of low-level quantum gates required to implement a sequence of high-level logical gates. 
However, they do not directly minimize this objective; instead, they minimize the number of high-level gates.

An optimization model was introduced to directly minimize the quantum cost in \cite{JungChoi2021-MultiCommodityNetwork}.
The authors use a multi-commodity flow-based model that is solved with a Mixed Integer Programming (MIP) solver.
The method is applied to functions with three to six qubits, and significant improvements in quantum cost between 18.8\% and 68.6\% are observed.
This provides a strong motivation to optimize quantum costs directly.
While good results are obtained for small functions, the method does not scale well beyond seven gates due to the exponential number of binary variables in the model.
Since the prior version of this manuscript, Jung and Choi \cite{jung2025new} built on their earlier model  \cite{JungChoi2021-MultiCommodityNetwork} and improved the results. However, this paper 
presents improvements that go well beyond \cite{jung2025new} in both solution quality and scalability. 
Specifically, the proposed approach successfully discovers new best-known circuits for several benchmark instances that were not identified by \cite{jung2025new}. 
Moreover, while the previous study addressed instances of up to six qubits or eight gates, this manuscript extends the scope to instances involving up to seven qubits and fifteen gates. 
This expansion is made possible by the enhanced computational efficiency of the formulation proposed in this work.

\subsection{Contributions}
This paper introduces a new optimization model to minimize quantum cost directly.
Compared to \cite{JungChoi2021-MultiCommodityNetwork} and \cite{jung2025new}, the new model is easier to implement, requires exponentially fewer binary variables, and has a beneficial block-angular structure.
Furthermore, the paper demonstrates the advantage of Constraint Programming (CP) in solving this model.
The key contributions can be summarized as follows:
\begin{itemize}
	\item The paper introduces a new optimization model and new symmetry-breaking constraints for MCT quantum circuit design.
	\item The new model allows both CP and MIP solvers to significantly improve solving time, with up to \emph{two orders of magnitude} speedup when the CP solver is used.
	\item Experiments with up to seven qubits and using up to 15 quantum gates result in several new best-known circuits for well-known benchmarks.
	\item An extensive comparison with other approaches shows that optimization models may require more time, but can provide superior circuits with guaranteed optimality.
    \item In-depth experiments and analyses demonstrate the benefit of the novel symmetry-breaking constraints.
\end{itemize}

\noindent
The remainder of the paper is organized as follows.
Section~\ref{sec:terminology} presents the necessary terminology and provides the problem description.
Section~\ref{sec:optimization-model} introduces the new optimization model, while Section~\ref{sec:symmetry-breaking-constraints} introduces new symmetry-breaking constraints for the target problem.
The computational results are presented by Section~\ref{sec:computational-experiments}, and Section~\ref{sec:conclusion} provides the conclusion.

\section{Terminology}
\label{sec:terminology}

As discussed in Section~\ref{sec:introduction}, this paper considers the design of quantum circuits for \emph{reversible Boolean functions} using \emph{MCT gates} to minimize the \emph{quantum costs} of the resulting circuit. 
The relevant definitions are introduced here, along with brief introduction into quantum computing.
Example~\ref{ex:example1} presents a running example that is used throughout this section.

\subsection{Basics of Quantum Computing}

\bmhead{Qubit and quantum state} 
A state of a quantum system is represented by \emph{qubits}, analogous to classical bits in classical computers. 
While bits assume values of 0 or 1 to define a single \emph{basis state} (commonly denoted by binary vector $|0\rangle = \begin{bmatrix} 1 & 0 \end{bmatrix}$ or $|1\rangle = \begin{bmatrix} 0 & 1 \end{bmatrix} $ in the context of quantum computing), qubits may represent a \emph{superposed state} (i.e., a complex vector $|\psi\rangle$) formed as a convex combination of the basis states.
In many quantum computers, qubits are implemented by various physical objects that can implement the superposed states, such as electrons, ions, photons, and superconducting circuits.
Algebraically, a quantum state of a qubit is represented as below.

\begin{equation}
|\psi\rangle  = \alpha |0\rangle + \beta |1\rangle 
= \alpha 
\begin{bmatrix} 1 \\ 0 \end{bmatrix} 
+ \beta 
\begin{bmatrix} 0 \\ 1 \end{bmatrix} 
 = 
\begin{bmatrix} \alpha \\ \beta \end{bmatrix},
\quad \alpha, \beta \in \mathbb{C}.
\label{eq:terminology_quantum-state}
\end{equation}
The qubit stores the likelihood of observing state 0 or 1 upon \textit{measurement}, while the measurement probability of each single state is computed by the 2-norm of each complex coefficient $\alpha$ and $\beta$, i.e., $|\alpha|^2$ and $|\beta|^2$.

\bmhead{Quantum gate} A \emph{quantum gate} operates on qubits to transition the system to a new state based on the specification.
Not every state transition can be realized by a single elementary gate, and multiple quantum gates may be combined into a \emph{quantum circuit} to represent more complicated functions.
These elementary quantum gates are realized through individual physical stimuli on corresponding qubits.

A single-qubit quantum gate is algebraically represented by a $2\times2$ unitary matrix $U$, that is, $U^{\dagger}U =UU^{\dagger} = I$, where $U^{\dagger}$ is the Hermitian conjugate.
\begin{equation}
  U |\psi \rangle = \alpha U|0 \rangle + \beta U |1 \rangle 
 = \begin{bmatrix} \alpha u_{11} + \beta u_{12} \\ \alpha u_{21} + \beta u_{22} \end{bmatrix}, \quad 
  U = \begin{bmatrix} u_{11} & u_{12} \\u_{21} & u_{22} \end{bmatrix}, u_{ij} \in \mathbb{C}.
\label{3_EqB_QuantumGate}
\end{equation}
Equation~(\ref{3_EqB_QuantumGate}) shows how a $2\times2$ unitary matrix $U$ transforms the given quantum state $|\psi\rangle $ from Equation~(\ref{eq:terminology_quantum-state}).
The matrix $U$ has four complex elements $u_{11}$, $u_{12}$, $u_{21}$ and $u_{22}$, where the first index denotes the row, and the second index denotes the column.
By multiplying $U$ to the state $|\psi\rangle $, $\alpha$ changes to $\alpha u_{11} + \beta u_{12}$, and $\beta$ to $\alpha u_{21}+\beta u_{22}$, respectively.
As a result, the measurement probability of each basis state has changed compared to the initial state.
Quantum gates may also be applied simultaneously to $k$ qubits at once, which corresponds to applying a $2^k \times 2^k$ matrix.

For the purpose of this paper, it is important to note from \eqref{eq:terminology_quantum-state} and \eqref{3_EqB_QuantumGate} that applying a quantum gate to a convex combination of basis states is equivalent to applying the quantum gate to each of the basis states and taking a convex combination of the results.
For the MCT quantum circuit design problem, it turns out to be sufficient to ensure that the circuit correctly transforms each basis state.
As a result, the remainder of this paper is completely discrete, and no complex numbers will be required.

\subsection{Reversible Boolean Function}
A reversible Boolean function is a bijective function where inputs and outputs are provided as binary strings of fixed length. 
This function is often presented in the form of a truth table.
It is noteworthy that reversible Boolean functions have been recognized as one of the fundamental operators in quantum computing, thus explored extensively in prior research on efficient quantum circuit synthesis \cite{saeedi2013synthesis}.  

Example~\ref{tab:example1tt} provides an example of a three-qubit reversible Boolean function. 
The specification defines a one-to-one mapping for each of the $2^3=8$ basis states. 
For instance, the input state (\texttt{qubit 1, qubit 2, qubit 3}) $ = (1, 1, 0)$ is mapped to the output state (\texttt{qubit 1, qubit 2, qubit 3}) $ = (0, 1, 1)$, or $110 \rightarrow 011$ for short.
It is sufficient to only specify the function for the basis states: when superposed states are involved, they can simply be decomposed into a convex combination of basis states, after which the function can be applied to each basis state according to the specification.

\begin{figure*}[t]
	\renewcommand\figurename{Example}
	\centering
	\begin{subfigure}{0.47\textwidth}
		\begin{tabular}{c|cc|c}
			\toprule
			Input & Output & Input & Output\\
			\cmidrule(lr){1-2} \cmidrule(lr){3-4}
			000 & 001 & 100 & 101\\
			001 & 000 & 101 & 100\\
			010 & 110 & 110 & 011\\
			011 & 111 & 111 & 010\\
			\bottomrule
		\end{tabular}
		\caption{Truth Table (completely specified).}
		\label{tab:example1tt}
	\end{subfigure}
	\begin{subfigure}{0.47\textwidth}
		\resizebox{0.73\textwidth}{!}{%
			\begin{quantikz}
				\lstick{\scriptsize $q=1$}	& \targ{}	\wire[l][1]["d=1"{above=0.2,pos=-0.3}]{a} & \ctrl{2} \wire[l][1]["d=2"{above=0.2,pos=-0.1}]{a}	&	\wire[l][1]["d=3"{above=0.2,pos=0}]{a}		& \\
				\lstick{\scriptsize $q=2$}	& \ctrl{-1}	& \control{}&			& \\
				\lstick{\scriptsize $q=3$}	& 			& \targ{}	&	\targ{}	&
			\end{quantikz}
		}
		\caption{Implementing Circuit.}
		\label{fig:example1c}
	\end{subfigure}
	\caption{Truth Table for Function Specification and Implementing Circuit (interactive: \href{https://algassert.com/quirk\#circuit=\%7B\%22cols\%22\%3A\%5B\%5B\%22Chance\%22\%2C\%22Chance\%22\%2C\%22Chance\%22\%5D\%2C\%5B\%22X\%22\%2C\%22\%E2\%80\%A2\%22\%5D\%2C\%5B\%22Chance\%22\%2C\%22Chance\%22\%2C\%22Chance\%22\%5D\%2C\%5B\%22\%E2\%80\%A2\%22\%2C\%22\%E2\%80\%A2\%22\%2C\%22X\%22\%5D\%2C\%5B\%22Chance\%22\%2C\%22Chance\%22\%2C\%22Chance\%22\%5D\%2C\%5B1\%2C1\%2C\%22X\%22\%5D\%2C\%5B\%22Chance\%22\%2C\%22Chance\%22\%2C\%22Chance\%22\%5D\%5D\%2C\%22init\%22\%3A\%5B1\%2C1\%5D\%7D}{\texttt{algassert.com/quirk}}).}
	\label{ex:example1}
\end{figure*}

\subsection{Multiple Control Toffoli (MCT) Circuit}

\emph{MCT circuits} consist of a sequence of \emph{MCT gates}.
Example~\ref{fig:example1c} provides a circuit that meets the specification of Example~\ref{tab:example1tt}. It has three horizontal lines (one for each qubit $q$) and three MCT gates (one per column $d$).
An MCT gate consists of one \emph{target qubit} with the $\oplus$ symbol and zero or more \emph{control qubits} with the \emph{$\bullet$} symbol.
Control qubits do not have to be adjacent, and vertical lines connect the control qubits to the target qubit.
For a given input, the circuit is read from left to right, and the MCT gates are applied one at a time.
Transitions follow the following rule: \emph{if all the control qubits are in state 1, then the target qubit is flipped}.

For example, consider the input $110$ and start from the very left of the given circuit Example~\ref{fig:example1c}.
The top line has state $1$, the middle line state $1$ and the bottom line state $0$.
The first gate has one control qubit on line two.
It follows that all control qubits are in state 1.
As a result, the target qubit (\texttt{qubit 1}) is flipped, changing the state to $010$ after the first gate.
The second gate has two control qubits, but they are not all in state 1 (\texttt{qubit 1} is in state 0) so nothing happens.
The third gate does not have any control bits, so the target qubit is flipped.
This results in the output $011$.
That is, the total transition is $110 \rightarrow 010 \rightarrow 010 \rightarrow 011$, meeting the specification.
It can be checked that the circuit meets the specifications for the other input states as well.

An important property of MCT circuits is that they are \emph{reversible}, i.e., they perform the inverse operation when read from right to left \cite{saeedi2010reversible}.
Therefore, MCT circuits are a natural candidate to represent reversible Boolean functions.
In fact, it is well-known that every reversible Boolean function can be represented in this way.

\subsection{Quantum Costs}
To implement an MCT circuit in practice, each MCT gate is decomposed into elementary quantum gates.
The number of elementary quantum gates is a well-established proxy for the cost of the MCT circuit, known as the \emph{quantum cost}.
Table~\ref{tab:quantum_cost} summarizes the best-known quantum cost $f(c)$ for an MCT gate that uses a total of $c \ge 0$ control qubits \cite{barenco1995elementary, maslov2003improved, grosse2009exact}.
Note that the costs change based on the number of \emph{slack qubits} that are available and that are not used in the MCT gate otherwise.
The cost of the circuit in Example~\ref{ex:example1} is $f(1) + f(2) + f(0) = 7$.
Note that the costs in the table may go down in the future as better decompositions are found.
It can also be seen that the quantum cost of an MCT gate tends to increase rapidly as more control qubits are added.

\begin{table}[t]
    \centering
	\begin{tabular}{c  @{\hspace{1.5em}} c @{\hspace{1.5em}} c @{\hspace{1.5em}} c @{\hspace{1.5em}} c @{\hspace{1.5em}} c @{\hspace{1.5em}} c @{\hspace{1.5em}} c @{\hspace{1.5em}} c}
    \toprule
    & \multicolumn{8}{c}{\textbf{Control qubits $p$}}   \\ \cmidrule{2-9}
    \textbf{Slack qubits}  & 0 & 1 & 2 & 3  & 4  & 5  & 6   & $\le$ 7 \\ \hline 
    0    & 1  & 1  & 5  & 13 & 29 & 62 & 125    &  $2^{p+1}-3$ \\ \hline
    1    & $\cdot$  & $\cdot$  & $\cdot$  & $\cdot$  & $\cdot$  & 52 & 80     &  $\cdot$ \\ \hline
    2    & $\cdot$  & $\cdot$  & $\cdot$  & $\cdot$  & 26 & $\cdot$  & $\cdot$      &  $\cdot$ \\ \hline
    3    & $\cdot$  & $\cdot$  & $\cdot$  & $\cdot$  & $\cdot$  & 38 & $\cdot$      &  $\cdot$ \\ \hline
$\le$ 4  & $\cdot$  & $\cdot$  & $\cdot$  & $\cdot$  & $\cdot$  & $\cdot$  & 50     &  $\cdot$  \\ 
    \bottomrule
    \end{tabular}
\vspace{0.5\baselineskip}
\caption{Quantum Costs for MCT Gates (dots indicate the same cost as above; slack qubits are qubits that are available but not used in the MCT gate).}%
\label{tab:quantum_cost}%
\vspace{-0.5\baselineskip}
\end{table}

\subsection {Remarks on Incomplete Specification}

In Example~\ref{ex:example1}, the truth table was completely specified, but this does not always have to be the case: 
depending on the application, there may be specific qubits that are used in the computation but for which the output is uninteresting (\emph{don't care} qubits).
However, every circuit \emph{implementation} still represents a bijective function that assigns specific states to the don't cares, due to the reversible nature of quantum operators.
Note that don't cares apply only to the outputs, whereas the inputs are completely specified in practice.

\begin{figure*}[t]
	\renewcommand\figurename{Example}
    \begin{subfigure}{0.7\textwidth}
		\begin{tabular}{c|ccc|cc}
			\toprule
			Input & Output & Impl. \ref*{fig:example2c} & Input & Output & Impl. \ref*{fig:example2c}\\
			\cmidrule(lr){1-3} \cmidrule(lr){4-6}
			000 & 00- & 001 & 100 & 101 & 101\\
			001 & 00- & 000 & 101 & 100 & 100\\
			010 & 11- & 111 & 110 & 011 & 011\\
			011 & \hspace{0.1em}-\hspace{0.1em}-\hspace{0.1em}- & 110 & 111 & 010 & 010\\
			\bottomrule
		\end{tabular}
		\caption{Incompletely Specified Truth Table.}
		\label{tab:example2tt}
	\end{subfigure}
	\begin{subfigure}{0.28\textwidth}
		\resizebox{0.9\textwidth}{!}{%
			\begin{quantikz}
				\lstick{\scriptsize $q=1$}	& \targ{} \wire[l][1]["d=1"{above=0.2,pos=-0.3}]{a}	& 	\wire[l][1]["d=2"{above=0.2,pos=0}]{a}		& \\
				\lstick{\scriptsize $q=2$}	& \ctrl{-1}	& 			& \\
				\lstick{\scriptsize $q=3$}	& 			& 	\targ{}	&
			\end{quantikz}
		}
		\caption{Implementing Circuit.}
		\label{fig:example2c}
	\end{subfigure}
	\caption{Truth Table of an Incompletely Specified Function and its Implementing Circuit (interactive: \href{https://algassert.com/quirk\#circuit=\%7B\%22cols\%22\%3A\%5B\%5B\%22Chance\%22\%2C\%22Chance\%22\%2C\%22Chance\%22\%5D\%2C\%5B\%22X\%22\%2C\%22\%E2\%80\%A2\%22\%5D\%2C\%5B\%22Chance\%22\%2C\%22Chance\%22\%2C\%22Chance\%22\%5D\%2C\%5B1\%2C1\%2C\%22X\%22\%5D\%2C\%5B\%22Chance\%22\%2C\%22Chance\%22\%2C\%22Chance\%22\%5D\%5D\%2C\%22init\%22\%3A\%5B1\%2C1\%5D\%7D}{\texttt{algassert.com/quirk}}).}
	\label{ex:example2}
\end{figure*}

Example~\ref{ex:example2} turns the complete specification of Example~\ref{ex:example1} into an incomplete specification by replacing some of the output qubits by ``-'' (don't care) instead of 0 or 1.
Note that the circuit in Example~\ref{fig:example1c} is still valid, but with the additional freedom, it might be possible to find a circuit with a cost lower than 7.
Example~\ref{fig:example2c} shows that a better circuit can be found.
The outputs of this implementation are added to Example~\ref{tab:example2tt} under ``Impl. \ref*{fig:example2c}''.
These outputs are different from the circuit in Example~\ref{ex:example1}, but they both meet the incomplete specification.
At a cost of only $f(1) + f(0) = 2$, the new circuit is a better implementation.

\section{Optimization Model}
\label{sec:optimization-model}

This section presents a new optimization model to design minimum-cost MCT quantum circuits that meet a given specification. The presentation assumes that the maximum number of gates is fixed and denoted by $m$. 
The new model is provided as Model~\eqref{form:qrcsp} throughout equations \eqref{eq:qrcsp:objective} -- \eqref{eq:qrcsp:flow_vars}, and the different components are introduced over the next several paragraphs.
Constraints~\eqref{eq:qrcsp:flip_implication}-\eqref{eq:qrcsp:keep_implication} are introduced as logical constraints first to clearly state their intent, after which a linear implementation is provided as one possible implementation.
An overview of the nomenclature is provided in Table~\ref{tab:symbols} for convenience.
While \cite{JungChoi2021-MultiCommodityNetwork} and this paper both use network flows to model quantum state transitions, there are significant differences between the models that are discussed at the end of this section.

\subsection{Base Network and Multicommodity Flows}

The new model uses network flows to model the state transitions caused by the circuit. Before formally introducing this part of the model, this paragraph aims to give some intuition through Figure~\ref{fig:network-example}.
The graph in this figure has vertices $(\sigma, d)$ that indicate being in state $\sigma \in \Omega$ before gate $d \in D = \{1,\ldots,m\}$ is applied.
The arrows show the transitions when state $010$ is provided as an input to Circuit~\ref{fig:example2c}.
That is, gate $d=1$ carries out the transition $010\rightarrow 110$ (vertex $(010, 1)$ to $(110, 2)$), and gate $d=2$ carries out the transition $110\rightarrow 111$ (vertex $(110, 2)$ to $(111, 3)$), for a total transition of $010\rightarrow 111$. 

Which state transitions are available depends on the design of the circuit. The dotted lines in Figure~\ref{fig:network-example} show the state transitions that are allowed when the circuit design variables represent Circuit~\ref{fig:example2c}.
Because of the reversible nature of MCT gates, the dotted lines create bijections between the states.
As such, flow entering vertex $(\sigma, 1)$ is pushed to a unique vertex $(\sigma', m+1)$, which represents the transition $\sigma \rightarrow \sigma'$ that is caused by the circuit.
To make sure that the state transitions match the specification, a source $S$ and a sink $T$ are added.
The source pushes the flow to a particular input state, while the sink only accepts flow from the correct output states.
In the example, input state $010$ has output specification $11-$ (Table~\ref{tab:example2tt}).
As such, the source pushes flow to $010$, and the sink only accepts flow from $110$ and $111$.
If no flow from $S$ to $T$ is possible, then the circuit fails to meet the specification, and the current assignment of the circuit design variables is infeasible.

A separate flow network is introduced for each input state to ensure that the circuit meets all the specifications.
The model is further improved by grouping input states with the same output specification into \emph{commodities}.
For a completely specified function, each commodity is associated with a single input state. 
In contrast, for an incomplete function, multiple input states can be grouped together based on the output specification.
The source pushes a unit flow to each of the input states in the group.
The flows remain separated due to the bijections, and are eventually collected by the sink.
Note that this grouping is only possible when input states have the same output specification, and therefore the arcs connecting to the sink are the same.

\begin{figure}[t]
	\includegraphics[width=\textwidth]{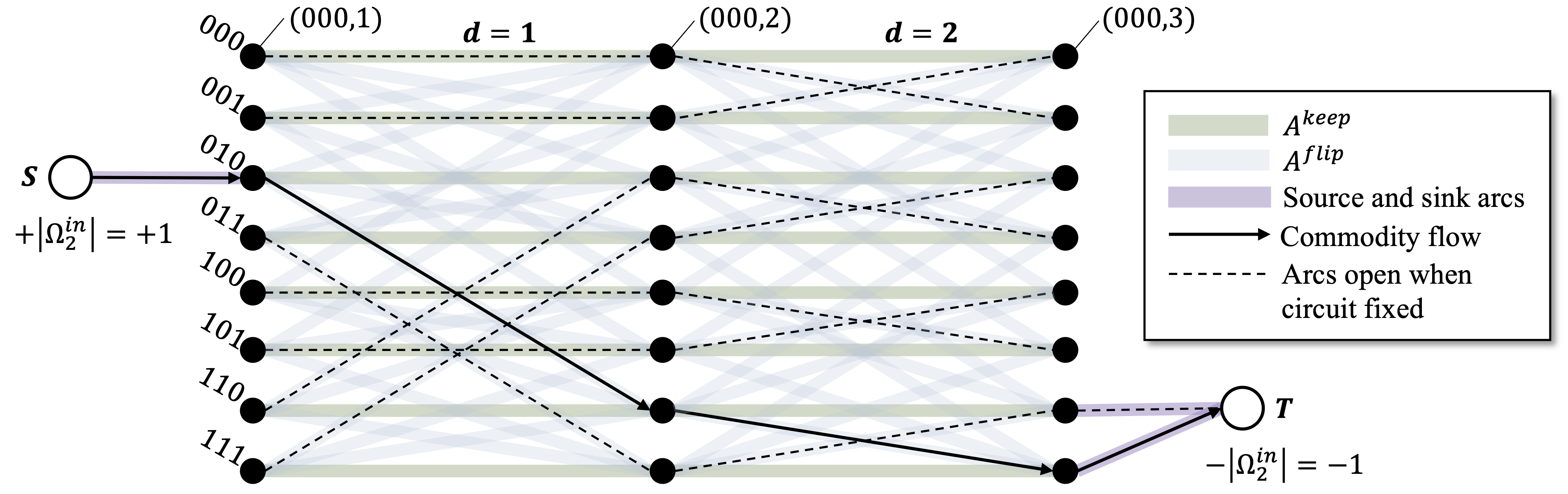}
	\caption{Flow Network Example for $k=2$ corresponding to Example~\ref{ex:example2}.}
	\label{fig:network-example}
\end{figure}

The set $\Omega$ of basis states is given by the $2^n$ binary vectors of length $n$.
For example, $\Omega = \{000, 001, 010, 011, 100, 101, 110, 111\}$ in Example~\ref{ex:example2}.
States with the same (possibly incomplete) output specification will be modeled together and grouped into commodities $k \in K$.
Each commodity is represented by the set $\Omega^\text{in}_k \subseteq \Omega$ of corresponding input states.
Again, using Example 2: $\Omega^\text{in}_1 = \{000,001\}$ (output specification $00-$), $\Omega^\text{in}_2 = \{010\}$ (output specification $11-$), etc.
The set $\Omega^\text{out}_k$ defines the permitted output states for each commodity $k \in K$.
In the example, this yields $\Omega^\text{out}_1 = \{000, 001\}$ (both match $00-$), $\Omega^\text{out}_2 = \{110, 111\}$ (both match $11-$), etc.
Note that the sets $\Omega^\text{in}_k$ partition $\Omega$ by definition, while the sets $\Omega^\text{out}_k$ may overlap and together cover $\Omega$.

Next, flow networks are defined so they can model the state transitions throughout the circuit.
Figure~\ref{fig:network-example} continues to provide a running example for this paragraph.
A graph $G_k = (V, A_k)$ is defined for each commodity $k\in K$.
The vertex set $V$ consists of a source $S$, a sink $T$, and the vertices $(\sigma, d)$ for $\sigma \in \Omega$, $d \in D \cup \{m+1\}$.
Each vertex $(\sigma, d)$ represents being in state $\sigma$ right after gate $d-1$ or right before gate $d$, i.e., gate $d$ is to be applied next.
The arc set $A_k$ of commodity $k \in K$ consists of four components:
\begin{itemize}
	\item \emph{Source arcs}: A set of arcs that connect the source node to the commodity input states, $(S, (\sigma, 1))$ $\forall \sigma \in \Omega^\text{in}_k$.
	\item \emph{Sink arcs}: A set of arcs that connect permitted output states to the sink node, $((\sigma, m+1), T)$, $\forall \sigma \in \Omega^\text{out}_k$.
	\item \emph{Flip arcs}: A set of arcs $A^\text{flip}_k$ that represent the cases when gate $d \in D$ acts on state $\sigma \in \Omega$ by flipping bit $q \in Q$. Let $\sigma \oplus q$ denote state $\sigma$ with bit $q$ flipped. Then $A^\text{flip}_k$ consists of arcs $((\sigma, d), (\sigma \oplus q, d+1))$, $\forall d \in D, \sigma \in \Omega, q \in Q$.
	\item \emph{Keep arcs}: A set of arcs $A^\text{keep}_k$ that represent the cases when gate $d \in D$ keeps state $\sigma \in \Omega$ the same even after the gate operation. That is, the keep arcs are given by $((\sigma, d), (\sigma, d+1))$, $\forall d\in D, \sigma \in \Omega$.
\end{itemize}

\noindent As shown in the Figure~\ref{fig:network-example}, each of four arc types are mutually exclusive and collectively exhaustive at the same time.
The formulation presented in the following section is developed based upon the network and the associated notations introduced above.

\subsection{Formulation}
The next paragraphs specify and explain Model~\eqref{form:qrcsp}.

\begin{mini!}
%
	{}
%
	{\sum_{d \in D} \sum_{j \in Q} f(j-1) y_j^d, \label{eq:qrcsp:objective}}
%
	{\label{form:qrcsp}}
%
	{}
%
%
	\addConstraint
	{t_q^d + w_q^d}
	{\le 1}
	{\forall q \in Q, d \in D, \label{eq:qrcsp:target_or_control}}
	\addConstraint
	{\sum_{q \in Q} t_q^d}
	{\le 1}
	{\forall d \in D, \label{eq:qrcsp:one_target}}
	\addConstraint
	{w_q^d}
	{\le \sum_{r \in Q} t_r^d}
	{\forall q \in Q, d \in D, \label{eq:qrcsp:control_needs_target}}
	\addConstraint
	{\sum_{j \in Q} j y^d_j}
	{= \sum_{q \in Q} t_q^d + \sum_{q \in Q} w^d_q}
	{\forall d \in D, \label{eq:qrcsp:def_sum_control1}}
	\addConstraint
	{\sum_{j \in Q} y^d_j}
	{\le 1}
	{\forall d \in D, \label{eq:qrcsp:def_sum_control2}}
	\addConstraint
	{\sum_{a \in \delta^+_k(v)} x^k_a - \sum_{a \in \delta^-_k(v)} x^k_a}
	{= \begin{cases}
			\lvert \Omega^\text{in}_k \rvert & \textrm{ if } v = S\\
			-\lvert \Omega^\text{in}_k \rvert & \textrm{ if } v = T\\
			0 & \textrm { else}
	\end{cases} \quad}
	{\forall k \in K, v \in V, \label{eq:qrcsp:flow_balance}}
	\addConstraint
	{\hspace{3.4cm}\mathllap{\left( \bigvee_{q^0 \in Q^0_{\sigma(a)}} \left(w_{q^0}^{d(a)} = 1\right)\right) \bigvee \left(t_{q(a)}^{d(a)} = 0\right) \Longrightarrow x_a^k = 0}}
	{}
	{\forall k \in K, a \in A^\text{flip}_k,\label{eq:qrcsp:flip_implication}}
	\addConstraint
	{\hspace{3.4cm}\mathllap{\left( \bigwedge_{q^0 \in Q^0_{\sigma(a)}} \left(w_{q^0}^{d(a)} = 0\right) \right) \bigwedge \left(\sum_{q \in Q} t_q^{d(a)} = 1\right) \Longrightarrow x_a^k = 0}}
	{}
	{\forall k \in K, a \in A^\text{keep}_k, \label{eq:qrcsp:keep_implication}}
 	\addConstraint
 	{t_q^d, w_q^d, y_j^d}
 	{\in \{0,1\}}
 	{\forall q, j \in Q, d \in D, \label{eq:qrcsp:design_vars}}
 	\addConstraint
	{x_a^k}
	{\in \{0,1\}}
	{\forall k \in K, a \in A_k. \label{eq:qrcsp:flow_vars}}
\end{mini!}

\begin{table}
\centering
\begin{tabular}{p{0.06\textwidth}p{0.84\textwidth}}
	\toprule
	Symbol & Definition \\
	\midrule
	\multicolumn{2}{l}{\textbf{Circuit Design: \labelcref{eq:qrcsp:target_or_control}-\labelcref{eq:qrcsp:control_needs_target}, \labelcref{eq:qrcsp:design_vars}}} \\[.05cm]
	$Q$ & $=\{1, \hdots, n\}$ set of qubits.\\
	$D$ & $=\{1, \hdots, m\}$ set of gates.\\
	$t_q^d$ & variable with value 1 if qubit $q \in Q$ is the target qubit of gate $d \in D$, and 0 otherwise.\\
	$w_q^d$ & variable with value 1 if qubit $q \in Q$ is a control qubit of gate $d \in D$, and 0 otherwise.\\
	\midrule
	\multicolumn{2}{l}{\textbf{Quantum Cost: \labelcref{eq:qrcsp:objective}, \labelcref{eq:qrcsp:def_sum_control1}-\labelcref{eq:qrcsp:def_sum_control2}, \labelcref{eq:qrcsp:design_vars}}} \\[.05cm]
	$f(c)$ & quantum cost of a single MCT gate with $c\ge 0$ control qubits.\\
	$y_j^d$ & variable with value 1 if gate $d \in D$ consists of a total of $j \in Q$ target and control qubits, zero otherwise.\\
	\midrule
	\multicolumn{2}{l}{\textbf{Quantum States and Flow Commodities: 
    \labelcref{eq:qrcsp:flow_balance}-\labelcref{eq:qrcsp:keep_implication}, \labelcref{eq:qrcsp:flow_vars}}} \\[.05cm]
	$\Omega$ & $=\{0_{(2)}, \hdots, (2^n-1)_{(2)}\}$ set of pure quantum states.\\
	$Q^0_\sigma$ & $= \{q \in Q : \sigma_q = 0\}$ set of qubits that are zero in state $\sigma \in \Omega$.\\
	$K$ & set of indices of the flow commodities; each commodity represents a set of input quantum states that have the same (possibly incomplete) output specification.\\
	$\Omega^\text{in}_k$ & $\subseteq \Omega$ set of input quantum states that represent commodity $k \in K$; together the sets $\Omega^\text{in}_k$ $\forall k\in K$ provide a partition of $\Omega$.\\
	$\Omega^\text{out}_k$ & $\subseteq \Omega$ set of quantum states that meet the (possibly incomplete) output specification associated with commodity $k \in K$; the sets $\Omega^\text{out}_k$ may overlap, and together cover $\Omega$.\\
	\midrule
	\multicolumn{2}{l}{\textbf{Flow Networks: \labelcref{eq:qrcsp:flow_balance}-\labelcref{eq:qrcsp:keep_implication}, \labelcref{eq:qrcsp:flow_vars}}} \\[.05cm]
	$V$ & set of vertices in each flow network; consists of source $S$, sink $T$, and nodes $(\sigma, d)$ $\forall \sigma \in \Omega, d \in D \cup \{m+1\}$.\\
	$A_k$ & set of arcs in the flow network of commodity $k\in K$.\\
	$A^\text{flip}_k$ & $\subset A_k$ set of arcs for commodity $k\in K$ that represent a transition that flips a qubit.\\
	$A^\text{keep}_k$ & $\subset A_k$ set of arcs for $k\in K$ that represent a transition that keeps the state the same.\\
	$x_a^k$ & variable with value 1 if commodity $k \in K$ uses arc $a \in A_k$, and 0 otherwise.\\
	$\delta^+_k(v)$ & $\subseteq A_k$ set of arcs for $k\in K$ coming out of vertex $v \in V$.\\
	$\delta^-_k(v)$ & $\subseteq A_k$ set of arcs for $k\in K$ coming into vertex $v \in V$.\\
	$d(a)$ & $\in D$ shorthand for the gate associated with arc $a \in A^\text{flip}_k \cup A^\text{keep}_k$.\\
	$q(a)$ & $\in Q$ shorthand for the qubit that is flipped by arc $a \in A^\text{flip}_k$.\\
	$\sigma(a)$& $\in \Omega$ shorthand for the state that arc $a \in A^\text{flip}_k \cup A^\text{keep}_k$ transitions from.\\
	\bottomrule
\end{tabular}
\caption{Nomenclature}
\label{tab:symbols}
\end{table}

\bmhead{Circuit Design}
The design of a quantum circuit should suffice for the necessary rules associated with the gate configuration of the selected gate library. 
This model imposes this relation through variables and constraints that describe quantum circuit diagrams presented in Example~\ref{fig:example1c} or Example~\ref{fig:example2c}.
The set of qubits $Q=\{1,\hdots,n\}$ describes the rows, while the set of gates $D=\{1, \hdots, m\}$ describes the columns.
Note that the maximum number of gates $m$ is an input to the problem.
Variables $t_q^d$ and $w_q^d$ are binary variables in Model~\eqref{form:qrcsp} that indicate whether $(q,d)$, contains a target or control qubit, respectively, for any $q\in Q$, $d\in D$.
The associated binary variables are defined in Equation~\eqref{eq:qrcsp:design_vars}.
Constraints~\eqref{eq:qrcsp:target_or_control} state that each spot defined by a combination of qubit $q$ and gate $d$, be a target qubit ($\oplus$), a control qubit ($\bullet$), or neither; but not both.
Constraints~\eqref{eq:qrcsp:one_target} enforce at most one target qubit per gate.
Gates without a target qubit are forced to be empty through Constraints~\eqref{eq:qrcsp:control_needs_target}.

\bmhead{Quantum Cost}
An MCT gate with $c \ge 0$ control qubits incurs a quantum cost of $f(c)$, as defined in Table~\ref{tab:quantum_cost}.
Let $y_j^d$ be a binary indicator that takes value one if gate $d \in D$ contains exactly $j \in Q$ target and control qubits, or zero otherwise, as defined in Equation~\eqref{eq:qrcsp:design_vars}.
Objective~\eqref{eq:qrcsp:objective} then calculates the total quantum cost over all gates.
Note that the input $q-1$ subtracts the single target qubit, as $f(\cdot)$ is defined in terms of control qubits only.
When gate $d\in D$ is empty, all indicators $y_j^d$ are zero (note $j \ge 1$) and there is no contribution to the objective.
The $y$-variables are forced to take the correct values by Constraints~\eqref{eq:qrcsp:def_sum_control1} and \eqref{eq:qrcsp:def_sum_control2}.
The former ensures that the indicators together represent the total number of target and control qubits, and the latter ensures that at most one indicator is active.

\bmhead{Flow Networks}
The state transitions for commodity $k\in K$ are represented by a network flow in $G_k$.
Equation~\eqref{eq:qrcsp:flow_vars} defines flow variables $x_a^k$ that take the value one if arc $a\in A_k$ is used, and zero otherwise.
The flow goes from the source to the sink of size $\lvert \Omega^\text{in}_k \rvert$. This flow is distributed to all the input states $\Omega^\text{in}_k$ by the source arcs, after which the flip and keep arcs model the state transitions.
The only way to reach the sink node is through the sink arcs that start from one of the permitted output states $\Omega^\text{out}_k$.
Equation~\eqref{eq:qrcsp:flow_balance} enforces that the $x_a^k$ variables represent such a flow.
Here $\delta^+_k(v)$ and $\delta^-_k(v)$ denote the out-arcs and in-arcs of vertex $v \in V$, respectively.

It remains to connect the circuit design decisions to the eligible network flows, i.e., ensure that arcs can only be used if they match the given circuit specification.
For convenience of presentation, the following shorthands are used for properties of flip and keep arcs $a \in A^\text{flip}_k \cup A^\text{keep}_k$: $d(a) \in D$ is the gate associated with arc $a$, $q(a) \in Q$ is the qubit that is flipped by arc $a$ (flip arcs only), and $\sigma(a) \in \Omega$ is the state that the arc $a$ transitions from.
Furthermore, for a given state $\sigma \in \Omega$ it will be convenient to define $Q^0_{\sigma}$ as the set of qubits that are in state 0, e.g., $Q^0_{010} = \{1,3\}$.
Constraints~\eqref{eq:qrcsp:flip_implication} and \eqref{eq:qrcsp:keep_implication} eliminate the flow from all flip and keep arcs that do not match the given circuit specification.
This can be seen by considering the outgoing arcs of an arbitrary vertex $(\sigma, d)$, $\sigma \in \Omega$, $d \in D$.
That is, the set of one keep arc and $n$ flip arcs that model the state transition due to gate $d$.
\begin{itemize}
	\item Case 1: \emph{Gate $d$ flips some qubit $\bar{q} \in Q$.} Based on the transition rule (Section~\ref{sec:terminology}), this means that $\bar{q}$ is the target qubit $(t_{\bar{q}}^d = 1)$ and all controls are on qubits with value 1 in state $\sigma$. Or alternatively, none of the controls are on qubits with value 0 in state $\sigma$ ($w_{q^0}^d = 0$ $\forall q^0 \in Q^0_{\sigma}$). It follows that the antecedent of \eqref{eq:qrcsp:keep_implication} holds and that the keep arc is excluded as expected. Flip arcs $a$ are eliminated by \eqref{eq:qrcsp:flip_implication} as soon as they flip the wrong qubit, i.e., $q(a) \neq \bar{q}$, which implies $t_{q(a)}^d = 0$. The arc that flips $\bar{q}$ is the only flip arc that is \emph{not} excluded by \eqref{eq:qrcsp:flip_implication}, as the target is in the right place ($t^d_{q(a)} = 1$) and all controls are on the zero states ($w_{q^0}^d = 0$).
	\item Case 2: \emph{Gate $d$ keeps state $\sigma$ the same.} Either there is no target qubit, in which case all flip arcs have $t_{q(a)}^d = 0$ and get eliminated by \eqref{eq:qrcsp:flip_implication} while the keep arc is unaffected by \eqref{eq:qrcsp:keep_implication}. Or there is a target qubit $\bar{q}$, but at least one of the controls is on a zero state, i.e., $w_{q^0}^d = 1$ for some $q^0 \in Q^0_\sigma$. It follows again that all flip arcs are eliminated while the keep arc is unaffected.
\end{itemize}

\noindent It is concluded that equations \eqref{eq:qrcsp:flip_implication}-\eqref{eq:qrcsp:keep_implication} close the correct arcs to match the circuit design.
It should also be noted that the flow variables may be relaxed to the continuous domain $x_a^k \in [0,1]$ $\forall k \in K, a \in A_k$.
This follows from the fact that for fixed values of the $t$, $w$, and $y$-variables, the remaining problem decomposes into $\lvert K \rvert$ independent minimum-cost flow problems, which are known to have the integrality property \cite{AhujaEtAl1993-NetworkFlowsTheory}.
This means that the new model requires only $\mathcal{O}(nm)$ binary variables.

\subsection{Implementation of Logical Constraints}

There are multiple ways to implement Constraints~\eqref{eq:qrcsp:flip_implication}-\eqref{eq:qrcsp:keep_implication}, depending on the solver.
This paper reformulates the implications as linear constraints that are widely supported:
\begin{subequations}
\begin{alignat}{2}
	x_a^k & \le t^{d(a)}_{q(a)} & \quad & \forall k \in K, a \in A^\text{flip}_k, \label{eq:linearized-logical-constr-1}\\
	x_a^k & \le 1-w^{d(a)}_{q^0} && \forall k \in K, a \in A^\text{flip}_k, q^0 \in Q^0_{\sigma(a)}, \label{eq:linearized-logical-constr-2}\\
	x_a^k & \le 1 - \sum_{q \in Q} t_q^{d(a)} + \sum_{q^0 \in Q^0_{\sigma(a)}} w^{d(a)}_{q^0} && \forall k \in K, a \in A^\text{keep}_k. \label{eq:linearized-logical-constr-3}
\end{alignat}
\label{eqs:linearized}
\end{subequations}

\noindent These specific constraints are justified by Proposition~\ref{prop1} below.
The reader is referred to \cite{Williams2013-ModelBuildingMathematical} for more general techniques to convert logical constraints.\\

\begin{proposition}
\label{prop1}
In Model~\eqref{form:qrcsp}, the logical constraints \eqref{eq:qrcsp:flip_implication}-\eqref{eq:qrcsp:keep_implication} may be implemented through linear constraints \eqref{eq:linearized-logical-constr-1}-\eqref{eq:linearized-logical-constr-3}.
\end{proposition}

\begin{proof}
Given a specific flow commodity $k \in K$, the logical constraints apply to two mutually exclusive cases: $a \in A_k^\text{flip}$ is a flip arc or $a \in A_k^\text{keep}$ is a keep arc. \\

\noindent \emph{Case 1 (Flip Arc):}
Arc $a \in A_k^\text{flip}$ is a flip arc, so implication \eqref{eq:qrcsp:flip_implication} applies.
If the antecedent holds, then it must be that $w_{q^0}^{d(a)} = 1$ for some $q^0 \in Q_{\sigma(a)}^0$ or that $t_{q(a)}^{d(a)} = 0$.
In this case the consequent is that $x_a^k = 0$.
Indeed, this is enforced by the linear constraints: Constraint~\eqref{eq:linearized-logical-constr-1} forces $x_a^k$ to zero if $t_{q(a)}^{d(a)} = 0$ and Constraints~\eqref{eq:linearized-logical-constr-2} do the same if any of the $w_{q^0}^{d(a)}$ are equal to one.
If the antecedent does not hold, then it must be that $w_{q^0}^{d(a)} = 0$ for all $q^0 \in Q_{\sigma(a)}^0$ and $t_{q(a)}^{d(a)} = 1$.
In this case, $x_a^k$ must not be constrained, and indeed, Constraints~\eqref{eq:linearized-logical-constr-1}-\eqref{eq:linearized-logical-constr-2} admit both $x_a^k = 0$ and $x_a^k = 1$.\\

\noindent \emph{Case 2 (Keep Arc):}
Arc $a \in A^\text{keep}_k$ is a keep arc, so implication \eqref{eq:qrcsp:keep_implication} applies.
If the antecedent holds, then $\bigwedge_{q^0 \in Q^0_{\sigma(a)}} \left(w_{q^0}^{d(a)} = 0\right)$ directly implies $\sum_{q^0 \in Q^0_{\sigma(a)}} w^{d(a)}_{q^0}=0$.
Furthermore, it must be that $\sum_{q \in Q} t_q^{d(a)} = 1$.
In this case, $x_a^k$ should be forced to zero.
Constraint~\eqref{eq:linearized-logical-constr-3} correctly implements this behavior by forcing $x_a^k \le 0$.
If the antecedent does not hold, then $w_{q^0}^{d(a)} = 1$ for some $q^0 \in Q_{\sigma(a)}^0$, which implies $\sum_{q^0 \in Q^0_{\sigma(a)}} w^{d(a)}_{q^0} \ge 1$, or  $\sum_{q \in Q} t_{q(a)}^{d(a)} = 0$.
In this case, $x_a^k$ should be not restricted.
If  $\sum_{q \in Q} t_{q(a)}^{d(a)} = 0$, then the right-hand side of Constraint \eqref{eq:qrcsp:one_target} is at least one and $x_a^k$ is indeed not restricted.
If $\sum_{q \in Q} t_{q(a)}^{d(a)} \ge 1$, Constraint~\eqref{eq:qrcsp:one_target} (at most one target bit) implies that the sum is exactly one.
At the same time, $\sum_{q^0 \in Q^0_{\sigma(a)}} w^{d(a)}_{q^0} \ge 1$, such that again the right-hand side is not restrictive.
It is concluded that both $x_a^k = 0$ and $x_a^k=1$ are admitted, as expected, which concludes the proof.

\end{proof}

\subsection{Comparison to the Previous Study}
The new optimization model differs from \cite{JungChoi2021-MultiCommodityNetwork} and \cite{jung2025new} in a number of significant ways.
First, \cite{JungChoi2021-MultiCommodityNetwork} separately treats four cases for each state transition: empty gate (no flip), zero control qubits (flip), one or more control qubits with flip, one or more control qubits without flip.
The new model captures all these cases in the same framework, thereby significantly simplifying the formulation and enhancing the computational efficiency.

Another important difference is in how the circuit is connected to opening and closing the flow arcs.
\cite{JungChoi2021-MultiCommodityNetwork} define a binary variable for each state $\sigma \in \Omega$ and gate $d\in D$ that identifies whether this state is modified by the gate, which introduces $\mathcal{O}(2^nm)$ binary variables.
Similar variables are used by \citep{jung2025new} as well.
The new model provides a much more direct way to close arcs through Constraints~\eqref{eq:qrcsp:flip_implication} and \eqref{eq:qrcsp:keep_implication}.
Compared to \cite{JungChoi2021-MultiCommodityNetwork}, this multiplies the number of constraints by a factor of $\mathcal{O}(n^2)$, \emph{but no additional binary variables are necessary.}
As a result, the new model requires only $\mathcal{O}(nm)$ binary variables.
The number of flow variables remains at $\mathcal{O}(2^n n m \lvert K \rvert)$.
However, as mentioned before, the new model decomposes into smaller independent minimum-cost flow problems for a fixed design.
This implies that the formulation has a block-angular structure that may be exploited by decomposition methods in future work.

\section{Symmetry-Breaking Constraints}
\label{sec:symmetry-breaking-constraints}
It has been observed previously that the optimal quantum circuit design is not necessarily unique.
The paper \cite{JungChoi2021-MultiCommodityNetwork} observes that empty gates do not affect the overall circuit, and constraints are added to force empty gates to appear at the end.
Reference \cite{IwamaEtAl2002-TransformationRulesDesigning} presents multiple transformations that lead to new circuits with equivalent outputs (but not necessarily the same quantum cost).
Inspired by these observations, this section defines three different swap operations -- \emph{Swap 1}, \emph{Swap 2}, and \emph{Swap 3} -- that result in a different but equivalent circuit with the same cost.

It will be proven that repeatedly applying these operations eventually results in a circuit that is \emph{unswappable}, i.e., no further swaps of these types can be applied.
This makes it possible to introduce symmetry-breaking constraints that limit the search to unswappable circuits, without loss of optimality.
After all, every swappable quantum circuit is associated with an unswappable quantum circuit that has the same cost.
The swaps used in this paper are introduced below, and future works may expand the list to eliminate additional symmetries.

\begin{figure}[t]
	\renewcommand\figurename{Example}
	\centering
	\begin{subfigure}{0.27\textwidth}
		\centering
		\begin{quantikz}[row sep=0.3cm]
			\lstick{\scriptsize $q=1$}	& \ctrl{3}		&			&	\\
			\lstick{\scriptsize $q=2$}	&				& \targ{}	&	\\
			\lstick{\scriptsize $q=3$}	& \control{}	& \ctrl{-1}	&	\\
			\lstick{\scriptsize $q=4$}	& \targ{}		&			&	
		\end{quantikz}
		\caption{Before Swap 2.}
		\label{fig:sym:dif:1}
	\end{subfigure}%
	\begin{subfigure}{0.22\textwidth}
		\centering
		\begin{quantikz}[row sep=0.3cm]
			&			& \ctrl{3}		&	\\
			& \targ{}	&				& 	\\
			& \ctrl{-1}	& \control{}	&	\\
			&			& \targ{}		&
		\end{quantikz}
		\caption{After Swap 2.}
		\label{fig:sym:dif:2}
	\end{subfigure}
	\begin{subfigure}{0.24\textwidth}
		\centering
		\begin{quantikz}[row sep=0.36cm]
			& \ctrl{2}	& 				&	\\
			& 			& \ctrl{2}		& 	\\
			& \targ{}	& \targ{}		&	\\
			&			& \control{}	&
		\end{quantikz}
		\caption{Before Swap 3.}
		\label{fig:sym:same:1}
	\end{subfigure}
	\begin{subfigure}{0.24\textwidth}
		\centering
		\begin{quantikz}[row sep=0.36cm]
			&				& \ctrl{2}		&	\\
			& \ctrl{2}		&				& 	\\
			& \targ{}		& \targ{}		&	\\
			& \control{}	& 				&
		\end{quantikz}
		\caption{After Swap 3.}
		\label{fig:sym:same:2}
	\end{subfigure}
	\caption{Swap Operations.}
	\label{ex:swap}
\end{figure}

\begin{itemize}
	\item \emph{Swap 1: Empty Gate.}
	If gate $d \in D$ is empty ($\sum_{q\in Q} t_q^d = 0$) and gate $d+1$ is not empty ($\sum_{q\in Q} t_q^{d+1} = 1$), then swap the two gates.

	\item \emph{Swap 2: Different Target.}
	If the target qubit $q \in Q$ of gate $d \in D$ is at a higher line than the target qubit $r\in Q$ of gate $d+1$ ($q > r$, lower in the diagram), and the target qubits do not neighbor a control qubit ($w_q^{d+1} = 0$ and $w_r^d = 0$), then swap the two gates.
	It was observed by \cite{IwamaEtAl2002-TransformationRulesDesigning} that the target qubits do not affect the control qubits in either direction, and hence the gates can safely be swapped.
	Example~\ref{fig:sym:dif:1}-\ref{fig:sym:dif:2} provides a visualization.

	\item \emph{Swap 3: Same Target.}
	If gate $d\in D$ and gate $d+1$ have the same target qubit $q \in Q$ ($t_q^d = 1$ and $t_q^{d+1} = 1$), and gate $d$ has fewer control bits ($\sum_{r \in Q} w_r^d < \sum_{r \in Q} w_r^{d+1}$), then swap the two gates.
	Again, the target qubits do not affect the neighboring control qubits, which justifies the swap.
	A visualization is provided by Example~\ref{fig:sym:same:1}-\ref{fig:sym:same:2}.
\end{itemize}

\begin{proposition}
	\label{prop:unswappable}
	Any swappable circuit can be turned into an unswappable circuit by repeatedly applying Swap 1-3.
\end{proposition}
\begin{proof}
	The number of permutations of the gates is finite, so it is sufficient to prove that the swaps can be applied in a way that avoids cycling.
	First repeatedly apply Swap 1, which results in all empty gates moving to the end of the circuit.
	These empty gates are not affected by Swap 2 and 3, and hence will stay in place and Swap 1 will not be applicable again.
	Let $\tau$ be a vector of length $m$ that contains the target qubit of each gate, or zero otherwise.
	Every Swap 2 strictly decreases $\tau$ in lexicographical order, which avoids cycling.
	For example, the swap in Example~\ref{fig:sym:dif:1}-\ref{fig:sym:dif:2} changes $\tau$ from $(4,2)$ to $(2,4)$.
	Whenever no Swaps 2 can be made, repeatedly applying Swap 3 has the effect of sorting groups with the same target qubit by the number of control qubits, and does not introduce cycles.
	Swap 3 does not affect the vector $\tau$, so no cycles are introduced when returning to Swap 2 after Swap 3 is exhausted.
	Eventually, an unswappable circuit is obtained after finitely many steps.
\end{proof}

\bmhead{Constraints}
Proposition~\ref{prop:unswappable} justifies symmetry-breaking constraints that enforce that the circuit is unswappable.
The three swaps are translated into the following three classes of inequalities:
\begin{subequations}
\begin{alignat}{2}
	\sum_{q \in Q} t_q^d & \ge \sum_{q \in Q} t_q^{d+1} & \quad & \forall d \in D, \label{eq:sym:empty_gate}\\
	t_q^d + t_r^{d+1} & \le 1 +  w_q^{d+1} + w_r^d & & \forall d \in D, q, r \in Q, q > r, \label{eq:sym:different_target}\\
	\sum_{r \in Q} w_r^d - \sum_{r \in Q} w_r^{d+1} &\ge (n-1)(t_q^d + t_q^{d+1} - 2) & & \forall d \in D, q \in Q. \label{eq:sym:same_target}
\end{alignat}
\end{subequations}
Constraints~\eqref{eq:sym:empty_gate} prevent Swap 1 by forcing gate $d \in D$ to be non-empty when gate $d+1$ is non-empty.
A similar constraint is used in \cite{JungChoi2021-MultiCommodityNetwork}.
Constraints~\eqref{eq:sym:different_target} model that if the target bits are set up correctly for Swap 2 (left-hand side of the equation equals two), then $w_q^{d+1} = 1$ or $w_r^d = 1$.
This is necessary because $w_q^{d+1} = w_r^d = 0$ would allow Swap 2 to be applied.
The constraint is automatically satisfied when the left-hand side is less than two.
Last, consider Constraints~\eqref{eq:sym:same_target}.
If the gates have targets on the same qubit ($t_q^d = t_q^{d+1} = 1$) then the constraint reduces to $\sum_{r \in Q} w_r^d \ge \sum_{r \in Q} w_r^{d+1}$ to prevent Swap 3. The term $n-1$ (the maximum number of control bits per gate) is sufficiently large to make the constraint inactive.

\section{Computational Experiments}
\label{sec:computational-experiments}
Computational experiments on well-known benchmarks are presented to demonstrate the performance of the new optimization model.
The new model and the symmetry-breaking constraints are implemented in Python 3.11 and run on a Linux machine with dual Intel Xeon Gold 6226 CPUs (24 cores in total) on the PACE Phoenix cluster \cite{PACE2017-PartnershipAdvancedComputing}.
CP-SAT 9.8.3296 \cite{ortools} is used as the CP solver with 24 workers (threads), and Gurobi 11.0.0 is used for the MIP approach.
The instances are sourced from RevLib \cite{WGT+:2008}, a common benchmark for reversible and quantum circuit design.
This paper selects Boolean functions with up to seven qubits that have known circuit implementations in fewer than 100 gates.
After removing easy three-qubit functions, a benchmark suite of 49 functions remains.
A time limit of 3600 seconds per instance is imposed in all experiments.
The circuit design problem is considered to be \emph{solved} if an optimal circuit implementation is found and proven to be optimal, or if it is proven that the problem is infeasible.
If no circuit is found, or if optimality cannot be proven within the time limit, then the time limit is reported as the runtime.

\subsection{Performance New Optimization Model}
This section compares the performance of the new optimization model to the optimization model of \cite{JungChoi2021-MultiCommodityNetwork}.
To the best of our knowledge, \cite{JungChoi2021-MultiCommodityNetwork} is the first MIP approach to the MCT quantum circuit design problem and therefore is closest to the current work.
The results are compared for the 38 reversible functions that overlap with the benchmark suite used in this paper, and experiments are carried out for the $m \in \{ 6, 7, 8 \}$ number of gates.

Table~\ref{tab:new_model_performance} demonstrates that the new optimization model systematically outperforms previous work.
Even accounting for the difference in hardware (6 cores vs. 24 cores), the new model is an order of magnitude faster when solved with Gurobi.
When solved with CP-SAT, the new runtimes even improve by another magnitude.
For $m=6$ example, CP is over 500x times faster than the runtime reported in \cite{JungChoi2021-MultiCommodityNetwork}.
It can also be seen that the new model solves significantly more instances.
All benchmarks with $m=7$ gates can now be solved, where only 20 out of 38 were solved previously.
Further inspection of the results reveals that every instance that is solved by \cite{JungChoi2021-MultiCommodityNetwork} is solved by the new model as well, regardless of whether MIP or CP was used.
When CP fails to solve the problem (which only happens for $m=8$), the best feasible solution is never worse than the best feasible solution obtained by \cite{JungChoi2021-MultiCommodityNetwork}.
The best performance is clearly obtained by the new model with CP, and this is the setting that is used for the remainder of the experiments.

\begin{table}[!t]
\centering
	\begin{tabular}{lrrrrrrr}
		\toprule
		& \multicolumn{4}{c}{Average Runtime (s)}
        & \multicolumn{3}{c}{Solved Instances}  \\
		\cmidrule(lr){2-5} \cmidrule(lr){6-8}
		& $m=6$ & $m=7$ & $m=8$ & Limit
        & $m=6$ & $m=7$ & $m=8$\\
		\midrule
		\cite{JungChoi2021-MultiCommodityNetwork} (MIP)
        & 6,614 & 21,126 & 29,895 & 36,000 
        & 36/38 & 20/38 & 7/38 \\
		New Model (MIP) 
        & 160 & 1252 & 2541 & 3,600 
        & 38/38 & 38/38 & 15/38 \\
		New Model (CP) 
        & 12 & 115 & 1193 & 3,600 
        & 38/38 & 38/38 & 28/38 \\
		\bottomrule
	\end{tabular}
	\vspace{0.5\baselineskip}
	\caption{Performance New Optimization Model compared to \cite{JungChoi2021-MultiCommodityNetwork}.}
	\label{tab:new_model_performance}
\end{table}

\subsection{Larger Instances}
The improved performance of the new optimization model motivates experiments on the full benchmark suite for $m=6$ up to $m=15$ gates, the larger instances that have not been solved in the previous MIP-based study \cite{JungChoi2021-MultiCommodityNetwork}.
This includes reversible functions such as \texttt{rd53} and \texttt{decod24-enable} that were not previously considered, and the number of gates that far exceeds the $m=8$ gates in the experiments by \cite{JungChoi2021-MultiCommodityNetwork}.

\begin{table}[!t]
\centering
	\begin{tabular}{lrrrrr}
		\toprule
		 & $m=6$ & $m=7$ & $m=8$ & $m=9$ & $m=10$ \\
		\cmidrule{2-6}
		Average Runtime (s) &    14 &   111 &  1,101 &  2,140 & 2,502 \\
		Solved Instances & 49/49 & 49/49 & 37/49 & 23/49 & 18/49 \\
        \midrule
        & $m=11$ & $m=12$ & $m=13$ & $m=14$ & $m=15$ \\
        \cmidrule{2-6}
		Average Runtime (s) & 2,754 & 2,757 & 2,753 & 2,761 & 2,758 \\
		Solved Instances & 13/49 & 13/49 & 13/49 & 13/49 & 13/49 \\
		\bottomrule
	\end{tabular}
	\vspace{0.5\baselineskip}
	\caption{Performance New Optimization Model with CP on Large Instances.}
	\label{tab:larger_instances}
	\vspace{-0.5\baselineskip}
\end{table}

\begin{figure}[t]
	\includegraphics[width=\textwidth]{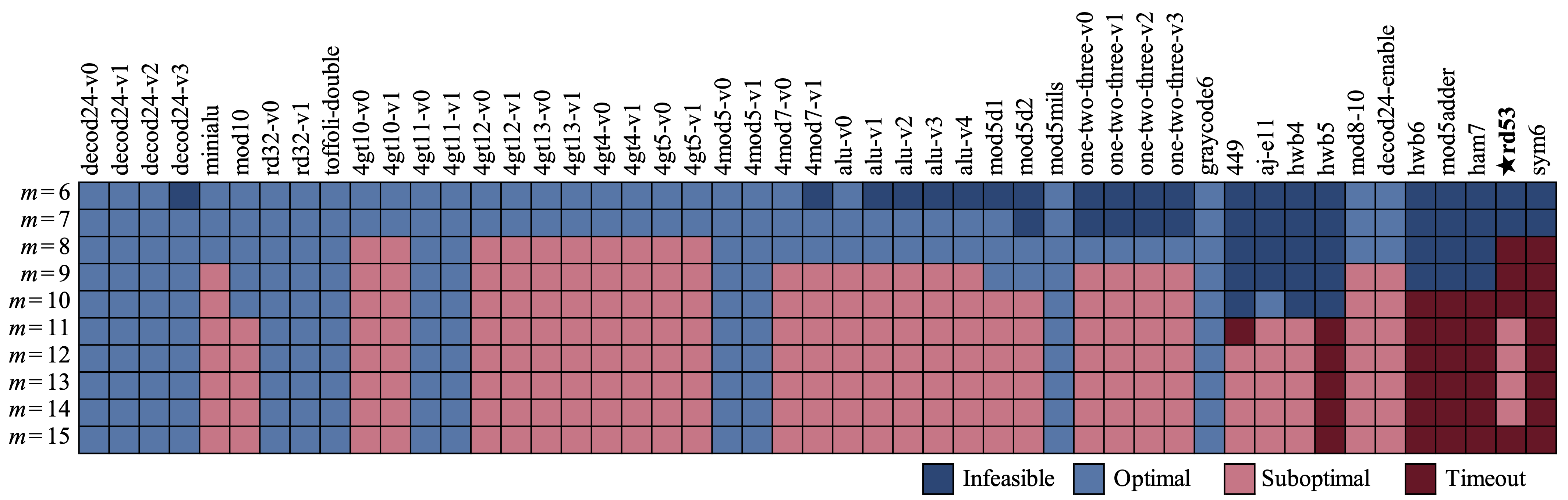}
	\vspace{-\baselineskip}
	\caption{Solution Status of New Model with CP at 1-Hour Time Limit.}
	\label{fig:instance_compare}
\end{figure}

Table~\ref{tab:larger_instances} shows that a big step was made in handling larger instances.
All instances with up to $m=7$ gates can now be solved in a matter of minutes on average, including the more complicated circuits.
While most instances with $m=8$ gates can still be solved within one hour, the average runtime starts to rise sharply at this point.
It is clear that more work remains to be done to solve the largest instances, although 13 out of the 49 instances with $m=15$ gates can already be solved.

Figure~\ref{fig:instance_compare} details the final solution status (infeasible, optimal, suboptimal, or timeout) for each of the large instances.
The newly added reversible functions are \texttt{449} up to \texttt{sym6} on the right side of the figure.
It is interesting to observe that these functions require a relatively large number of gates.
For example, \texttt{rd53} (marked with a $\star$) is proven infeasible for $m=6$ and $m=7$, and the solver was unable to find a feasible solution for $m \in \{8,9,10\}$, which are presumably infeasible.
At $m=11$, however, the solver fails to prove optimality but does find a feasible circuit.
This circuit, shown in Figure~\ref{fig:rd53-new}, improves the state of the art with a quantum cost of 47, using only seven qubits.
This is a cost improvement of 28\% compared to the seven-qubit circuit presented in \cite{WGT+:2008}, for example, which has a quantum cost of 65.
This demonstrates the benefit of using a model that can handle a larger number of gates, as good results may be obtained even when optimality cannot be proven.

\begin{figure}[t]
\centering
\begin{tikzpicture}

  \node[anchor=north west, inner sep=0] (A) at (0,0) {
    \begin{subfigure}{0.48\textwidth}
      \centering
      \adjustbox{max width=\linewidth}{
        \begin{quantikz}[row sep=0.36cm]
          &         &\targ{}	 & 	        &\ctrl{5} &\targ{}   &\ctrl{4} &\ctrl{3} &\targ{}   &\ctrl{5} &        &\ctrl{6} 
          &\ctrl{6} &\targ{}   &\targ{}   &\ctrl{4} & \\
          &\ctrl{4} &\ctrl{-1} &\targ{}   &\ctrl{4} &          &         &         &          &         &\targ{} &\ctrl{5} 
          &\ctrl{5} &          &          &         & \\
          &         &          &\ctrl{-1} &         &\ctrl{-2} &         &         &          &         &        &\ctrl{4}         
          &         &          &          &         & \\
          &      	  &          &	        &         &          &         &\targ{}  &\ctrl{-3} &         &        &         
          &\ctrl{3} &\ctrl{-3} &\ctrl{-3} &         & \\
          &         &          &          &         &          &\targ{}  &         &\ctrl{-4} &         &        &         
          &\ctrl{2} &\ctrl{-4} &          &\targ{}  & \\
          & \targ{} &	         &          &\targ{}  &          &         &         &          &\targ{}  &        &         
          &         &          &          &         & \\ 
          &         & 	     &	        &         &          &         &         &          &         &        &\targ{}          
          &\targ{}  &          &          &         & \\
        \end{quantikz}
      }
      \caption{Previous circuit (quantum cost: 65) \cite{WGT+:2008}.}
      \label{fig:rd53-old}
    \end{subfigure}
  };

  \node[anchor=north west, inner sep=0] (B) at ([xshift=1.2cm]A.north east) {
    \begin{subfigure}{0.48\textwidth}
      \centering
      \adjustbox{max width=\linewidth}{
        \begin{quantikz}[row sep=0.22cm]
          & \ctrl{5}  &	\targ{}	 & 	         & \ctrl{5}   &            &            &             &           &             &             &          &  \\
          &		    &            &	         &            &            & \ctrl{5}   & \ctrl{4}    & \targ{}   &  \ctrl{5}   & \ctrl{4}    & \ctrl{3} &  \\
          &           &            &	\targ{}  & \control{} &  \targ{}   & \control{} & \control{}  & \ctrl{-1} &             &             &          &  \\
          &      		& \ctrl{-3}  &	         &            &  \ctrl{-1} &            &             &           &             &             &          &  \\
          &        	&            &           &         	  &            &            &             &           &  \control{} & \control{}  & \targ{}  &  \\
          & \targ{}   &	         & \ctrl{-3} & \targ{}	  &            & \control{} & \targ{}     &           &  \control{} & \targ{}     &          &  \\
          &       	& 	         &	         &            &            &   \targ{}  &             &           &  \targ{}    &             &          &  \\
        \end{quantikz}
      }
      \caption{New Circuit ($m = 11$, quantum cost: 47).}
      \label{fig:rd53-new}
    \end{subfigure}
  };

\draw[thin, black]
  ($(A.north east)!0.5!(B.north west)$) -- 
  ($($(A.south east)!0.5!(B.south west)$)+(0,5pt)$);

\end{tikzpicture}
\caption{Previous best known circuit and new circuit for \texttt{rd53}.}
\label{fig:rd53}
\end{figure}

\begin{figure}[t]
\centering
\begin{tikzpicture}

  \node[anchor=north west, inner sep=0] (A) at (0,0) {
    \begin{subfigure}{0.48\textwidth}
      \centering
      \adjustbox{max width=\linewidth}{
        \begin{quantikz}[row sep=0.2cm]
        & \targ{}   &          & \targ{}   & \ctrl{4} & \ctrl{2}  & \ctrl{3} &  \\
        & \ctrl{-1} & \ctrl{3} &           &          &           &          &  \\
        &           &          & \ctrl{-2} & \ctrl{2} & \targ{}   &          &  \\
        &           &          & \ctrl{-3} & \ctrl{1} & \ctrl{-1} & \targ{}  &  \\
        & \ctrl{-4} & \targ{}  & \ctrl{-4} & \targ{}  &           &          & 
        \end{quantikz}
      }
      \caption{Previous circuit (quantum cost: 38) \cite{WGT+:2008}.}
      \label{fig:4mod7-v0-old}
    \end{subfigure}
  };

  \node[anchor=north west, inner sep=0] (B) at ([xshift=1.2cm]A.north east) {
    \begin{subfigure}{0.48\textwidth}
      \centering
      \adjustbox{max width=\linewidth}{
        \begin{quantikz}[row sep=0.36cm]
        & \targ{} &        &        & \targ{}   & \ctrl{4}  & \targ{}   & \ctrl{4}  & \targ{} &\ctrl{2}& \ctrl{3}  &   \\
        &		    & \ctrl{3}   & \targ{}   & \ctrl{-1} &           &           &           &           &           &           &   \\
        &           &            & \ctrl{-1} &           &           & \ctrl{-2} &           &           & \targ{}   &           &   \\
        &      		&            &           &           &\ctrl{1}   &           &\ctrl{1}   &           &\ctrl{-1}  &  \targ{}   &  \\
        &\ctrl{-4}  &	\targ{}	 &           & \ctrl{-4} & \targ{}   & \ctrl{-4} & \targ{}   & \ctrl{-4} &           &           &   \\
        \end{quantikz}
      }
      \caption{New Circuit ($m = 10$, quantum cost: 30).}
      \label{fig:4mod7-v0-new}
    \end{subfigure}
  };

\draw[thin, black]
  ($(A.north east)!0.5!(B.north west)$) -- 
  ($($(A.south east)!0.5!(B.south west)$)+(0,5pt)$);

\end{tikzpicture}
\caption{Previous best known circuit and new circuit for \texttt{4mod7-v0}.}
\label{fig:4mod7-v0}
\end{figure}

\begin{figure}[t]
\centering
\begin{tikzpicture}

  \node[anchor=north west, inner sep=0] (A) at (0,0) {
    \begin{subfigure}{0.48\textwidth}
      \centering
      \adjustbox{max width=\linewidth}{
        \begin{quantikz}[row sep=0.36cm]
        &		   &          &\targ{}   &\ctrl{3}  &         &\ctrl{1}  &         &\ctrl{2}  &          &\ctrl{4} &         & \\
        &		   &\targ{}   &\ctrl{-1} &\ctrl{2}  &\ctrl{2} &\targ{}   &\ctrl{2} &          &\targ{}   &         &\ctrl{2} & \\
        &\targ{}   &          &          &          &\ctrl{1} &          &\ctrl{1} &\targ{}   &\ctrl{-1} &         &\ctrl{1} & \\
        &\ctrl{-1} &          &          &\targ{}   &\targ{}  &\ctrl{-2} &\targ{}  &\ctrl{-1} &\ctrl{-2} &\ctrl{1} &\targ{}  & \\
        &		   &\ctrl{-3} &\ctrl{-4} &\ctrl{-1} &         &\ctrl{-3} &         &\ctrl{-2} &          &\targ{}  &         & \\
        \end{quantikz}
      }
      \caption{Previous circuit (quantum cost: 40) \cite{WGT+:2008}.}
      \label{fig:one-two-three-v0-old}
    \end{subfigure}
  };

  \node[anchor=north west, inner sep=0] (B) at ([xshift=1.2cm]A.north east) {
    \begin{subfigure}{0.48\textwidth}
      \centering
      \adjustbox{max width=\linewidth}{
        \begin{quantikz}[row sep=0.22cm]
        &		    &           & \targ{}   & \ctrl{3}  &           &           &           &           &           & \\
        &		    &           &           &           &           & \targ{}   &\ctrl{2}   &           &\ctrl{3}   & \\
        & \targ{}   &           &           & \ctrl{1}  & \targ{}   & \ctrl{-1} &           & \targ{}   &           & \\
        &\ctrl{-1}  & \targ{}   &\ctrl{-3}  & \targ{}   &\ctrl{-1}  &           & \targ{}   & \ctrl{-1} &           & \\
        &		    &\ctrl{-1}  &           &           &           & \ctrl{-3} &           &           & \targ{}   & \\
        \end{quantikz}
      }
      \caption{New Circuit ($m = 9$, quantum cost: 17).}
      \label{fig:one-two-three-v0-new}
    \end{subfigure}
  };

\draw[thin, black]
  ($(A.north east)!0.5!(B.north west)$) -- 
  ($($(A.south east)!0.5!(B.south west)$)+(0,5pt)$);

\end{tikzpicture}
\caption{Previous best known circuit and new circuit for \texttt{one-two-three-v0}.}
\label{fig:one-two-three-v0}
\end{figure}

\begin{figure}[t]
\centering
\begin{tikzpicture}

  \node[anchor=north west, inner sep=0] (A) at (0,0) {
    \begin{subfigure}{0.48\textwidth}
      \centering
      \adjustbox{max width=\linewidth}{
        \begin{quantikz}[row sep=0.33cm]
        &		    &\targ{}    &           &\ctrl{4}  &          &          &          &\ctrl{3} &  \\
        &\targ{}    &           &           &          &          &\ctrl{3}  &          &         &  \\
        &		    &\ctrl{-2}  &\targ{}    &\ctrl{2}  &          &          &\targ{}   &\ctrl{1} &  \\
        &\ctrl{-2}  &           &\ctrl{-1}  &          &\targ{}   &\ctrl{1}  &\ctrl{-1} &\targ{}  &  \\
        &           &           &\ctrl{-2}  &\targ{}   &\ctrl{-1} &\targ{}   &          &         &  \\
        \end{quantikz}
      }
      \caption{Previous circuit (quantum cost: 24) \cite{WGT+:2008}.}
      \label{fig:one-two-three-v3-old}
    \end{subfigure}
  };

  \node[anchor=north west, inner sep=0] (B) at ([xshift=1.2cm]A.north east) {
    \begin{subfigure}{0.48\textwidth}
      \centering
      \adjustbox{max width=\linewidth}{
        \begin{quantikz}[row sep=0.36cm]
        &		    &           &           & \targ{}   &\ctrl{2}   &           &           &           &           & \\
        & \targ{}   &           &           &           &           &           &\ctrl{2}   &           &           & \\
        &		    &\ctrl{1}   & \targ{}   &\ctrl{-2}  & \targ{}   &           &\ctrl{1}   & \targ{}   &\ctrl{2}   & \\
        &		    & \targ{}   &           &           &\ctrl{-1}  & \ctrl{1}  & \targ{}   &\ctrl{-1}  &           & \\
        &\ctrl{-3}  &           &\ctrl{-2}  &           &           & \targ{}   &           &           & \targ{}   & \\
        \end{quantikz}
      }
      \caption{New Circuit ($m = 9$, quantum cost: 17).}
      \label{fig:one-two-three-v3-new}
    \end{subfigure}
  };

  \draw[thin, black]
    ($(A.north east)!0.5!(B.north west)$) -- 
    ($($(A.south east)!0.5!(B.south west)$)+(0,6pt)$);

\end{tikzpicture}
\caption{Previous best known circuit and new circuit for \texttt{one-two-three-v3}.}
\label{fig:one-two-three-v3}
\end{figure}

Figures~\ref{fig:rd53} to \ref{fig:one-two-three-v3} present both the previously best-known circuits specified in \cite{WGT+:2008} and the newly developed circuits, along with their corresponding quantum costs.
At first glance, it may seem that the quantum cost simply decreases as the number of MCT gates decreases.
In general, this trend holds true, as observed in the cases of \texttt{rd53} (Figure~\ref{fig:rd53}), and \texttt{one-two-three-v0} (Figure~\ref{fig:one-two-three-v0}).
However, quantum cost can also be significantly reduced even if the number of MCT gates increases, provided that the added MCT gates have fewer control bits.
Such behavior is illustrated in \texttt{4mod7-v0} (Figure~\ref{fig:4mod7-v0}) and \texttt{one-two-three-v3} (Figure~\ref{fig:one-two-three-v3}), where circuits with more MCT gates still achieve lower overall quantum cost.
These results highlight that even small changes to the circuit structure, particularly in the number of control bits in MCT gates, can lead to a substantial reduction in quantum cost.
This phenomenon arises due to the exponential scaling of the costs associated with MCT gates, as shown in Table~\ref{tab:quantum_cost} in Section~\ref{sec:terminology}.

\subsection{Effect of Symmetry-Breaking Constraints}

\begin{figure}[]
	\includegraphics[width=\textwidth]{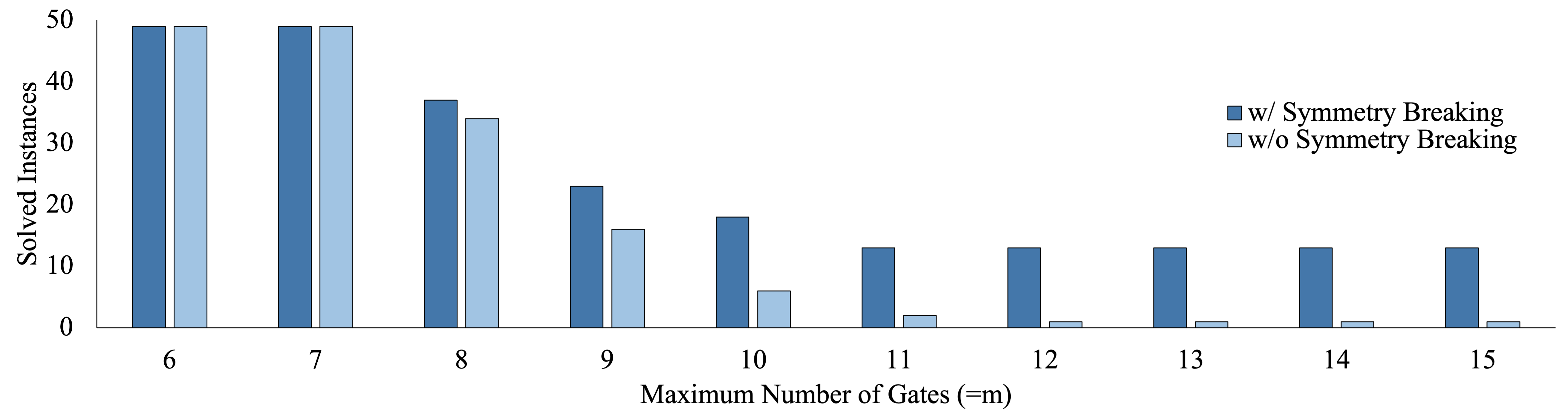}
	\caption{Comparison Symmetry-Breaking Constraints.}
	\label{fig:effect-symmetry}
\end{figure}

Figure~\ref{fig:effect-symmetry} demonstrates the benefit of using symmetry-breaking constraints~\eqref{eq:sym:empty_gate}-\eqref{eq:sym:same_target}.
Without symmetry-breaking constraints, all instances with $m=6$ and $m=7$ gates can still be solved, but the average solution times are 29\% and 139\% longer, respectively.
For $m\ge 8$ gates, the difference in solvability becomes apparent.
Out of the largest instances with $m=15$ gates, only \texttt{graycode6} can be solved without breaking symmetries, while 13 instances can be solved when the constraints are included.
It is not surprising that the symmetry-breaking constraints are more effective for longer circuits (i.e., large $m$ values), as they are expected to have more symmetric solutions.
More interestingly, adding the constraints outperforms the built-in symmetry detection in CP-SAT, which suggests that the symmetries observed in this paper are not obvious to detect automatically.

\begin{figure}[]
    \centering
	\includegraphics[width=\textwidth]{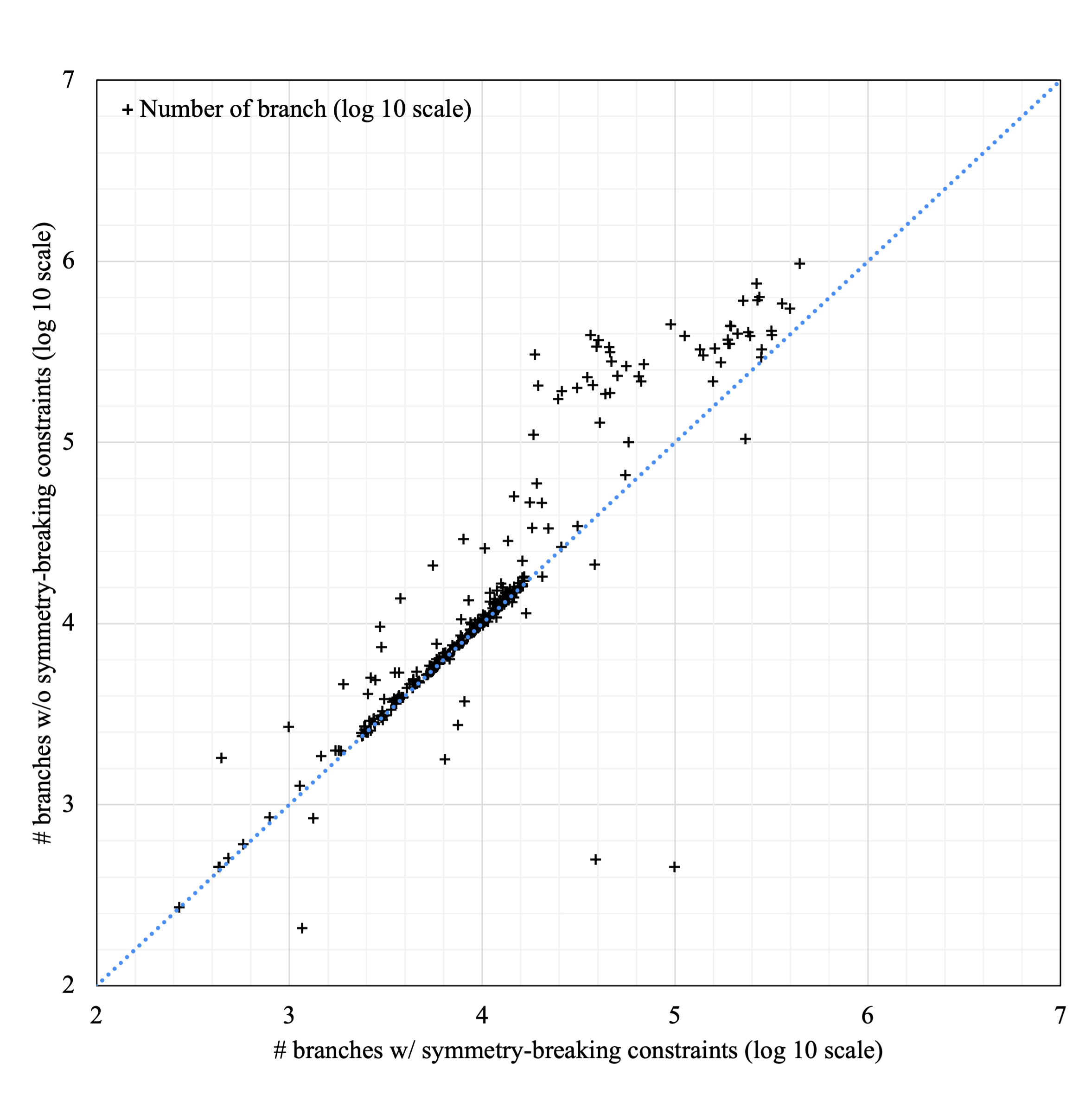}
	\caption{Branch Count Comparison Based on Symmetry-Breaking Constraints}
	\label{fig:symbrk-branches}
\end{figure}

The impact of the symmetry-breaking constraints also shows in the number of branches required to solve the given problem.
Figure~\ref{fig:symbrk-branches} compares the branch counts with and without symmetry-breaking constraints.
The horizontal axis shows the branch count when the symmetry-breaking constraints are applied, while the vertical axis presents the branch count without them.
Each datapoint corresponds to a single instance with a certain specification and number of gates $m$.
The plot uses log scales and the blue-dotted line indicates the diagonal.
If the point is located above this line, it means that the branch count is reduced when the symmetry-breaking constraints are incorporated into the optimization model.
It is clear that the symmetry-breaking constraints improve performance in almost all cases, with improvements of up to an order of magnitude.
This is particularly true for instances that previously required a large number of branches, i.e., instances that are difficult to solve without breaking symmetries.

\begin{table}[ht]
\centering
\begin{tabular}{clrrrrrr}
\toprule
 & & \multicolumn{2}{c}{\textbf{\#sol. w/o sym.}} & \multicolumn{2}{c}{\textbf{\#sol. w/ sym.}} & \multicolumn{2}{c}{\textbf{Reduced rate}} \\
\cmidrule(r){3-4} \cmidrule(r){5-6} \cmidrule(r){7-8}
\textbf{No.} & \textbf{Function} & Total & Optimal & Total & Optimal & Total (\%) & Optimal (\%) \\
\midrule
1 & toffoli-double    & 44778 & 40  & 14965 & 2  & 66.58 & 95.00 \\
2 & graycode6         & 46    & 6   & 9     & 1  & 80.43 & 83.33 \\
3 & mod10             & 11    & 8   & 9     & 6  & 18.18 & 25.00 \\
4 & 4mod7-v0          & 1     & 1   & 1     & 1  & 0.00  & 0.00  \\
5 & decod24-enable    & 1056  & 18  & 296   & 3  & 71.97 & 83.33 \\
6 & mod5mils          & 280   & 40   & 32    & 2  & 88.57 & 95.00  \\
7 & minialu           & 2862  & 33  & 1796  & 14 & 37.25 & 57.58 \\
8 & mod8-10           & 92    & 2   & 76    & 1  & 17.39 & 50.00 \\
9 & 4gt4-v0           & 219   & 1   & 63    & 1  & 71.23 & 0.00  \\
10 & rd32-v1          & 3464  & 148 & 747   & 30 & 78.44 & 79.73 \\
11 & 4gt10-v1         & 2368  & 5   & 537   & 2  & 77.32 & 60.00 \\
12 & decod24-v0       & 75    & 5   & 48    & 4  & 36.00 & 20.00 \\
13 & decod24-v1       & 3     & 2   & 2     & 1  & 33.33 & 50.00 \\
14 & decod24-v2       & 23    & 4   & 6     & 1  & 73.91 & 75.00 \\
\bottomrule
\end{tabular}
\vspace{0.5\baselineskip}
    \caption{Solution Count Comparison Based on Symmetry-Breaking Constraints.}
    \label{tab:feasible-solution-count}
\end{table}

Finally, the impact of the symmetry-breaking constraint can be examined by comparing the total number of feasible solutions, as well as the number of \emph{optimal} solutions, where optimal solutions here refer to circuits that achieve the optimal objective value.
Table~\ref{tab:feasible-solution-count} presents the total number of feasible solutions and optimal solutions that are identified with the \texttt{SearchForAllSolutions} feature of the CP solver, applied to the satisfiability version of the model with $m=6$ gates.
That is, the objective is either removed (for the total number of solutions) or a constraint is added to force the objective to take on the optimal value (for the number of optimal solutions).

Table~\ref{tab:feasible-solution-count} includes fourteen cases in which the solver successfully collected all feasible solutions within the specified computation time.
The values in the \emph{Reduced rate} column represent the percentage of solution count reduced by applying the symmetry-breaking constraints. 
Across all cases, the number of feasible solutions clearly decreases, from the minimum of 17.39\% to the maximum of 88.57\%.
An exceptional case of \texttt{4mod7-v0} does not show a change in the number of solutions because there is only one feasible solution in the feasible region.
Furthermore, in terms of the optimal solution count, eleven out of 14 cases show a reduction in the solution count when the symmetry-breaking constraints are applied.
As the number of feasible and optimal solutions in the feasible region decreases, the solver typically requires less time to find an optimal solution.
These results explicitly support the claim that the symmetry-breaking constraint enhances computational performance by effectively eliminating symmetric solutions from the feasible region.

\begin{figure}[t]
	\centering

	\begin{subfigure}{0.45\textwidth}
		\centering
		\adjustbox{max width=\linewidth}{
		\begin{quantikz}[row sep=0.4cm]
&\targ{}   &\targ{}   &          &         &          &\ctrl{4}  &  \\
&\ctrl{-1} &          &\targ{}   &\ctrl{3} &          &          &  \\
&          &\ctrl{-2} &          &\ctrl{2} &\targ{}   &\ctrl{2}  &  \\
&\ctrl{-3} &          &          &         &          &          &  \\
&          &          &\ctrl{-3} &\targ{}  &\ctrl{-2} &\targ{}   &  \\
		\end{quantikz}
		}
		\caption{Unswappable optimal circuit.}
		\label{fig:swappable-sols-1}
	\end{subfigure}
	\hspace{0.04\textwidth}
	\begin{subfigure}{0.45\textwidth}
		\centering
		\adjustbox{max width=\linewidth}{
		\begin{quantikz}[row sep=0.4cm]
&\targ{}   &\targ{}   &          &         &          &\ctrl{4}  &  \\
&          &\ctrl{-1} &\targ{}   &\ctrl{3} &          &          &  \\
&\ctrl{-2} &          &          &\ctrl{2} &\targ{}   &\ctrl{2}  &  \\
&          &\ctrl{-3} &          &         &          &          &  \\
&          &          &\ctrl{-3} &\targ{}  &\ctrl{-2} &\targ{}   &  \\
		\end{quantikz}
		}
		\caption{Swappable optimal circuit \#1.}
		\label{fig:swappable-sols-2}
	\end{subfigure}

	\begin{subfigure}{0.45\textwidth}
		\centering
		\adjustbox{max width=\linewidth}{
		\begin{quantikz}[row sep=0.4cm]
&\targ{}   &          &\targ{}   &         &          &\ctrl{4}  &  \\
&\ctrl{-1} &\targ{}   &          &\ctrl{3} &          &          &  \\
&          &          &\ctrl{-2} &\ctrl{2} &\targ{}   &\ctrl{2}  &  \\
&\ctrl{-3} &          &          &         &          &          &  \\
&          &\ctrl{-3} &          &\targ{}  &\ctrl{-2} &\targ{}   &  \\
		\end{quantikz}
		}
		\caption{Swappable optimal circuit \#2.}
		\label{fig:swappable-sols-3}
	\end{subfigure}
	\hspace{0.04\textwidth}
	\begin{subfigure}{0.45\textwidth}
		\centering
		\adjustbox{max width=\linewidth}{
		\begin{quantikz}[row sep=0.4cm]
&\targ{}   &          &         &\targ{}   &          &\ctrl{4}  &  \\
&\ctrl{-1} &\targ{}   &\ctrl{3} &          &          &          &  \\
&          &          &\ctrl{2} &\ctrl{-2} &\targ{}   &\ctrl{2}  &  \\
&\ctrl{-3} &          &         &          &          &          &  \\
&          &\ctrl{-3} &\targ{}  &          &\ctrl{-2} &\targ{}   &  \\
		\end{quantikz}
		}
		\caption{Swappable optimal circuit \#3.}
		\label{fig:swappable-sols-4}
	\end{subfigure}

	\caption{Swappable group of circuits found from the satisfiability version of the proposed optimization model for \texttt{4gt10-v1}, $m=6$.}
	\label{fig:swappable-sols}
\end{figure}

\begin{figure}
		\centering
		\adjustbox{max width=0.45\linewidth}{
		\begin{quantikz}[row sep=0.4cm]
&\targ{}   &          &\ctrl{4} &\targ{}   &          &\ctrl{4}  &  \\
&\ctrl{-1} &\targ{}   &\ctrl{3} &          &          &          &  \\
&          &          &         &\ctrl{-2} &\targ{}   &\ctrl{2}  &  \\
&\ctrl{-3} &          &         &          &          &          &  \\
&          &\ctrl{-3} &\targ{}  &          &\ctrl{-2} &\targ{}   &  \\
		\end{quantikz}
		}
		\caption{Unswappable optimal circuit from the satisfiability version of the proposed optimization model for \texttt{4gt10-v1}, $m=6$.}
		\label{fig:unswappable-sol}
\end{figure}

As a simple example, Figures~\ref{fig:swappable-sols} and \ref{fig:unswappable-sol} illustrate the optimal circuits discovered for function \texttt{4gt10-v1}, which corresponds to row 11 in Table~\ref{tab:feasible-solution-count}.
Without symmetry-breaking constraints, the solver finds all five optimal circuits shown in Figure~\ref{fig:swappable-sols} and Figure~\ref{fig:unswappable-sol}.
However, it is evident that the four circuits in Figure~\ref{fig:swappable-sols}, while structurally different, are considered symmetric according to the symmetries defined in Section~\ref{sec:symmetry-breaking-constraints}.

For example, Figure~\ref{fig:swappable-sols-2} has a Swap 3 into Figure~\ref{fig:swappable-sols-1} between the first and second gates, and Figure~\ref{fig:swappable-sols-3} has a Swap 2 into Figure~\ref{fig:swappable-sols-1} between the second and third gates.
Lastly, Figure~\ref{fig:swappable-sols-4} has a Swap 2 into Figure~\ref{fig:swappable-sols-3} between the third and fourth gates. 
It can be checked that Figure~\ref{fig:swappable-sols-1} is an unswappable circuit.
Figure~\ref{fig:unswappable-sol} shows another solution that satisfies the input function specification but the gates are not swappable nor could not be derived from the group of circuits shown in Figure~\ref{fig:swappable-sols}.

In contrast, when symmetry-breaking constraints are applied, only two circuits—\ref{fig:swappable-sols-1} and \ref{fig:unswappable-sol}—are found.
This outcome confirms that the proposed symmetry-breaking constraints successfully eliminate symmetric feasible solutions, and more specifically, symmetric solutions within a group of optimal solutions, thereby reducing redundancy in the feasible solution space.

\subsection{Comparative Analysis}
A comparative analysis is provided to show how optimization-based methods fit in with other methods considered in the literature.
Papers are selected that synthesize the entire circuit from scratch (as opposed to post-processing), that report quantum cost and computation time for every experiment, and for which the benchmark suite overlaps significantly with the current paper.
This results in six studies that are summarized by Table~\ref{tab:comparing_papers}.
The papers provide a mix of exact and heuristic methods that provide different trade-offs in terms of solution time and solution quality.
Also note that while all papers report quantum cost, the methods themselves often use a different objective function as a proxy, such as minimizing the number of gates or the number of qubits.

\begin{table}][H]
\centering
\begin{tabular}{cccccccc} \toprule
  \textbf{Label} &
  \textbf{Method} &
  \textbf{Objective} &
  \textbf{Type} &
  \textbf{Gate Lib.} &
  \textbf{Max Time} &
   \\ \midrule
\cite{lin2014rmdds} &
  \begin{tabular}[c]{@{}c@{}}Reed-Muller \\ + decision diagram\end{tabular} &
  Gate count &
  Heuristic &
  MCT &
  600s &
   \\ \hline
\cite{krishna2014efficient} &
  \begin{tabular}[c]{@{}c@{}}Subgraph matching \\ + decision diagram\end{tabular} &
  Qubit count &
  Heuristic &
  \begin{tabular}[c]{@{}c@{}} MCT up to \\ two controls \end{tabular} &
  $<$1s  &
   \\ \hline
\cite{grosse2009exact} &
  \begin{tabular}[c]{@{}c@{}}Satisfiability\\ problem\end{tabular} &
  Gate count &
  Exact &
  MCT &
  5,000s &
   \\ \hline
\cite{wille2008quantified} &
  \begin{tabular}[c]{@{}c@{}}Quantified Boolean
  \\ satisfiability problem\end{tabular} &
  Gate count &
  Exact &
  MCT &
  2,000s &
   \\ \hline
\cite{JungChoi2021-MultiCommodityNetwork} &
  \begin{tabular}[c]{@{}c@{}}Optimization model \\ + MIP solver\end{tabular} &
  Quantum cost &
  \begin{tabular}[c]{@{}c@{}}Exact \end{tabular} &
  MCT &
  36,000s &
   \\ \hline
\cite{jung2025new} &
  \begin{tabular}[c]{@{}c@{}}Optimization model \\ + Branch-and-Cut \end{tabular} &
  Quantum cost &
  \begin{tabular}[c]{@{}c@{}}Exact \end{tabular} &
  MCT &
  7,200s &
   \\ \hline
  CP  &
  \begin{tabular}[c]{@{}c@{}}Optimization model\\ + CP solver\end{tabular} &
  Quantum cost &
  Exact &
  MCT &
  3,600s &
  \multicolumn{1}{l}{} \\
  \bottomrule
\end{tabular}%
\vspace{0.5\baselineskip}
\caption{Summary of Papers for Comparative Analysis.}
\label{tab:comparing_papers}
\end{table}

Figure~\ref{fig:CA_plots-1} and \ref{fig:CA_plots-2} includes 42 plots that show the performance of each method with a circle on the two-dimensional \emph{time-quantum cost} plane.
To generate these plots, this paper selects a solution with the lowest quantum costs regardless of whether optimality is proven.
The labels in the plot correspond to each study as presented in the \emph{Label} column of Table~\ref{tab:comparing_papers}.
The gray circles correspond to the solutions that have not been proven to be optimal, while the blue circles indicate that the solution was proven to be optimal in the selected $m$ in terms of quantum cost, which is only optimized directly by \cite{JungChoi2021-MultiCommodityNetwork}, \cite{jung2025new}, and this paper.
While the current paper only considers circuits for a given number of qubits, \cite{krishna2014efficient} and \cite{lin2014rmdds} introduce additional qubits to obtain a feasible design.
The number of qubits used is indicated by the relative size of the circle in Figure~\ref{fig:CA_plots-1} and \ref{fig:CA_plots-2}.
A circuit is considered better if it has a lower quantum cost, and methods are preferred when they have a shorter solution time and introduce fewer ancilla qubits.
That is, small circles on the lower left are preferred.

Figure~\ref{fig:CA_plots-1} and \ref{fig:CA_plots-2} show that the current method outperforms the other methods in quantum cost (or in solution time if the quantum cost is tied) in 25 of 42 cases.
These function names are marked by a blue box in the top right corner.
Some large instances, such as \texttt{aj-e11} and \texttt{rd53}, were not tackled by \cite{JungChoi2021-MultiCommodityNetwork} and \cite{jung2025new}, but the new model produces (near-)optimal solutions in the given time limit.
There are also several instances where the ability to handle an increased number of gates results in lower quantum costs, without sacrificing computation time (e.g. \texttt{one-two-three-v0}, \texttt{one-two-three-v3}, \texttt{4mod7-v0}).
Even when the new model does not reduce quantum cost, in many cases it can still provide \emph{guaranteed} optimality in a shorter computation time than other methods (e.g., \texttt{4gt10-v0}, \texttt{4gt10-v1}, \texttt{4gt12-v0}, \texttt{4gt12-v1}, \texttt{4gt4-v0}, \texttt{4gt13-v0}, \texttt{4gt5-v0}, \texttt{4gt5-v1}, \texttt{alu-v1}, \texttt{alu-v3}, \texttt{alu-v4}, \texttt{decod24-enable}, \texttt{decod24-v1}, \texttt{decod24-v2}, \texttt{minialu}, \texttt{mod10}, \texttt{mod5d2}, \texttt{one-two-three-v1}, \texttt{one-two-three-v2}, \texttt{rd32-v1}).

\thispagestyle{empty}
\vspace*{\fill}
\begin{center}
  \includegraphics[width=0.95\textwidth]{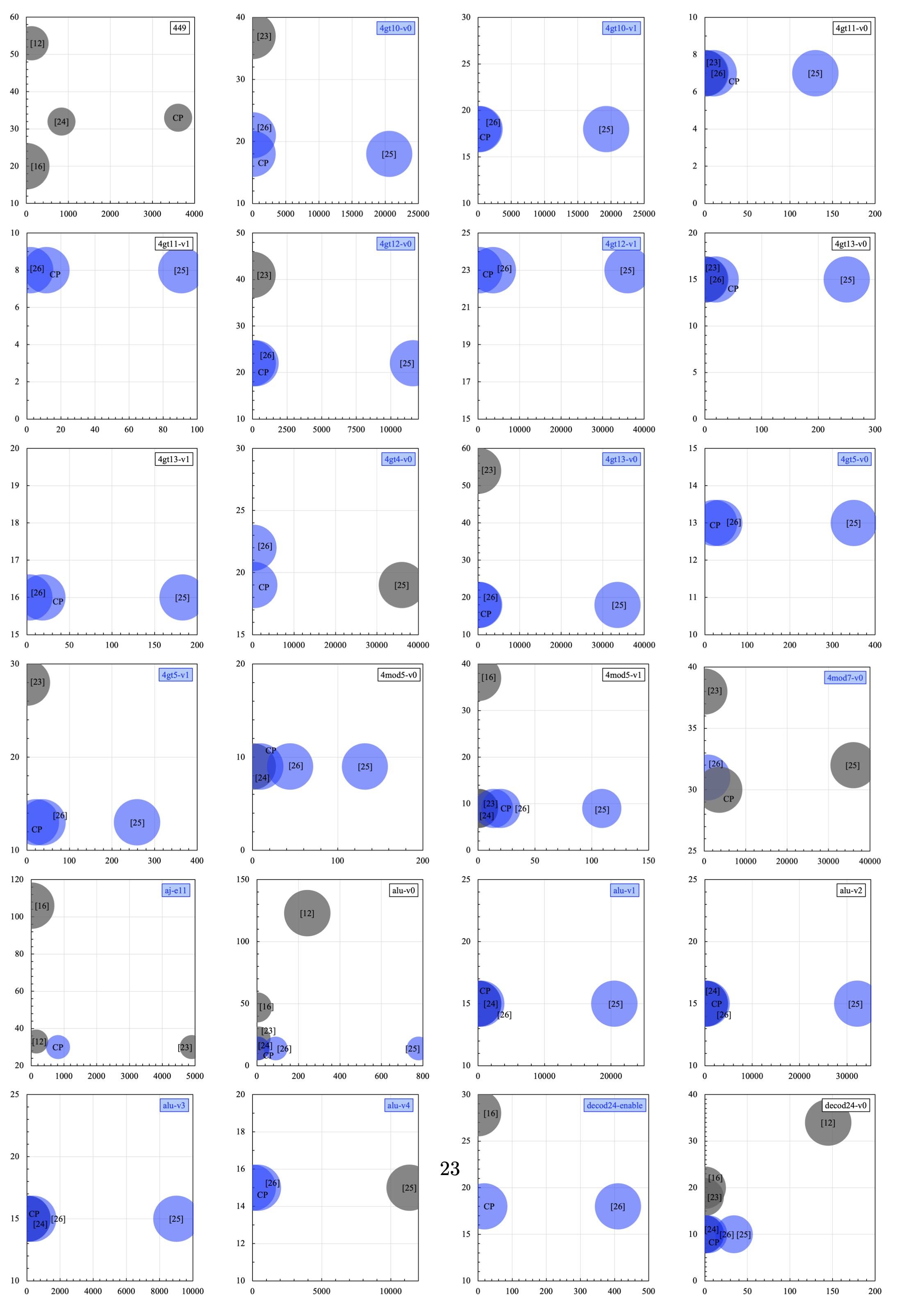}
  \vspace{1em}
  
  \captionof{figure}{Comparison of Best Results with Previous Studies.}
  \label{fig:CA_plots-1}
\end{center}
\vspace*{\fill}

\clearpage
\thispagestyle{empty}
\vspace*{\fill}
\begin{center}
  \includegraphics[width=0.95\textwidth]{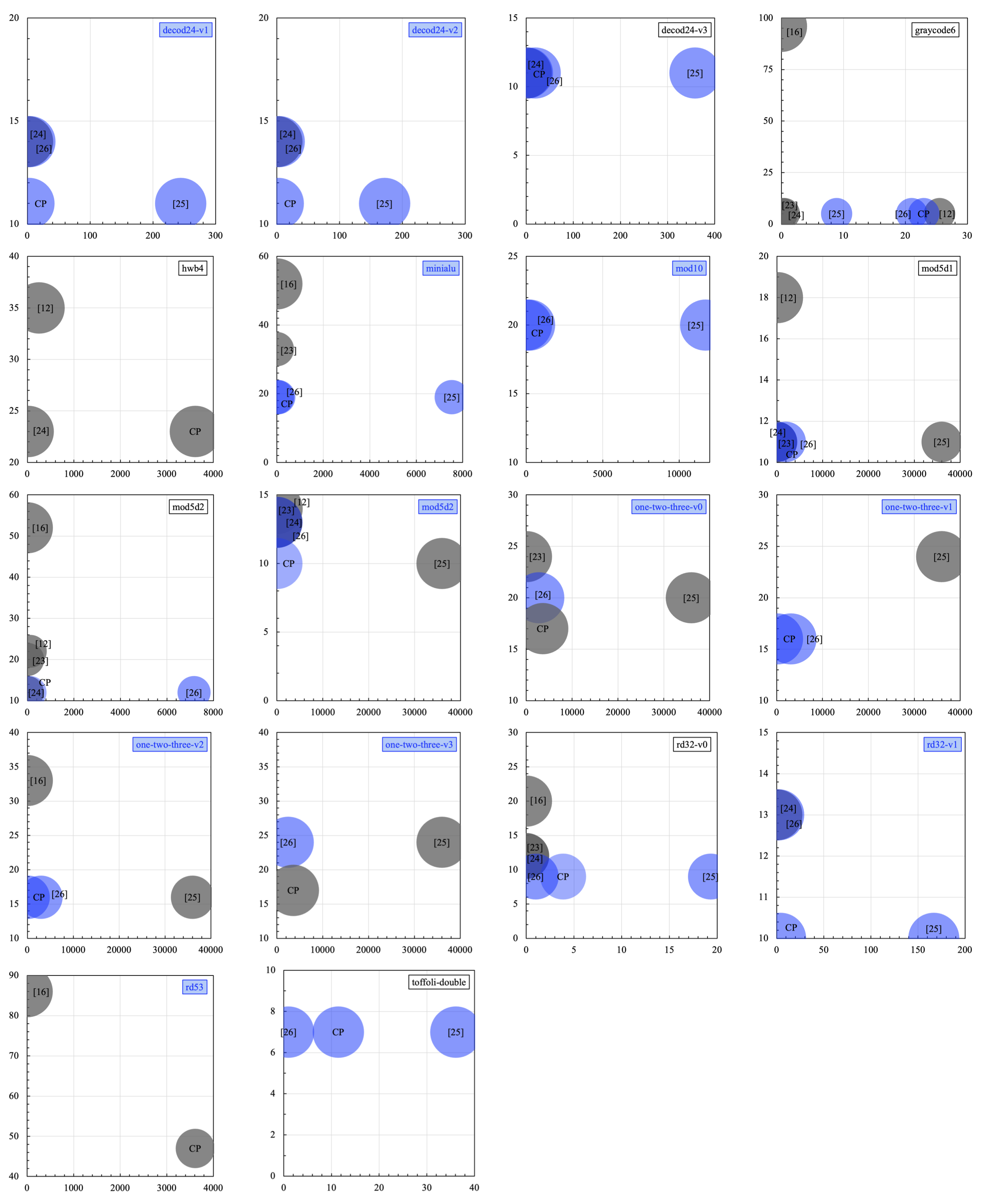}
  \vspace{1em}
  
  \captionof{figure}{Comparison of Best Results with Previous Studies.\emph{(cont.)}}
  \label{fig:CA_plots-2}
\end{center}
\vspace*{\fill}

\clearpage
\indent In general, the improved performance may come at the cost of a longer solution time, and the other methods may be better suited for very time-constrained environments.
In particular, \cite{wille2008quantified} provides high-quality solutions in a short amount of time, but improvements are still possible for some instances.
For example, for \texttt{decod24-v1}, \cite{wille2008quantified} presents a circuit with six gates and a quantum cost of 14.
This solution is found by first minimizing the number of gates, and then minimizing the quantum cost when the number of gates is fixed to six.
The new optimization model, however, can directly solve the case with seven gates to find a circuit with quantum cost 11.
This indicates that even for a small circuit of only four qubits, significant savings may still be obtained: a 21\% cost reduction in this case.

\section{Conclusion}
\label{sec:conclusion}

This paper introduces a new optimization model along with a set of symmetry-breaking constraints tailored for the design of quantum circuits using multiple-control Toffoli (MCT) gates.
The paper begins by outlining essential quantum computing concepts, thereby providing sufficient background for understanding the proposed work.
The newly developed optimization model simplifies the constraint structure compared to previous formulations, resulting in a substantial reduction in the number of binary variables required in the model.
This simplification not only enhances the tractability but also contributes to improved computational efficiency.
In addition, a set of symmetry-breaking constraints is proposed specifically to eliminate symmetric solutions caused by swappable gate pairs, thereby reducing the size of the feasible region without compromising the optimality of the solution.

Computational experiments demonstrate that the proposed model enables both CP and MIP solvers to solve instances significantly faster.
In particular, CP solvers achieve up to two orders of magnitude speedup on certain benchmark cases.
Furthermore, experiments on larger instances—featuring up to seven qubits and 15 gates—led to the discovery of four new quantum circuits that outperform the previously best-known solutions in terms of quantum cost.
The effectiveness of symmetry-breaking constraints is especially notable in large-scale instances.

To provide more insight, a series of in-depth analyses is conducted examining the change in branch counts and the number of feasible and optimal solutions.
These results validate the role of symmetry-breaking constraints in enhancing solver performance by guiding the search away from redundant solutions.
The paper also provides an illustrative example that shows how symmetric circuit structures are eliminated through the inclusion of the symmetry-breaking constraints.
Finally, a detailed comparison between different methodologies reveals an important trade-off: while optimization-based methods may require more computational time than heuristic or rule-based approaches, they offer the advantage of guaranteed optimality and can produce circuit designs of superior quality.

There are several promising directions for future work.
Although the proposed model is effective for circuits of moderate size, technical challenges remain in scaling the approach to handle circuits with a larger number of qubits and gates.
One promising approach is the application of decomposition techniques; the structure of the optimization model aligns itself naturally with decomposition, as the problem reduces to a collection of independent minimum-cost flow subproblems once the binary variables are fixed.
Another direction is to generalize the model to accommodate different gate libraries, or to move beyond MCT gates and directly optimize over circuits built from elementary quantum gates.
Such extensions would broaden the applicability of the model and further integrate optimization-based circuit synthesis into the broader landscape of quantum computing.

\backmatter

\bmhead{Acknowledgements}
This research was partly funded by the NSF AI Institute for Advances in Optimization (Award 2112533).

\section*{Declarations}
\begin{itemize}
\item Funding: This research was partly funded by the NSF AI Institute for Advances in Optimization (Award 2112533).
\item Conflict of interest/Competing interests: The authors declare no conflict of interest.
\item Ethics approval and consent to participate: Not applicable.
\item Consent for publication: Not applicable.
\item Data availability: The entire functions are available in \url{https://revlib.org} \cite{WGT+:2008}.
\item Materials availability: Not applicable.
\item Code availability: Code will be made available at the time of publication.
\item Author contribution: Conceptualization (J. Jung, K. Dalmeijer); Methodology (J. Jung, K. Dalmeijer, P. Van Hentenryck); Software (J. Jung); Validation, Formal Analysis, Investigation (J. Jung, K. Dalmeijer, P. Van Hentenryck); Data Curation (J. Jung); Writing - Original Draft (J. Jung, K. Dalmeijer); Writing - Review \& Editing (K. Dalmeijer, P. Van Hentenryck); Visualization (J. Jung); Supervision (K. Dalmeijer, P. Van Hentenryck); Project Administration (P. Van Hentenryck); Funding Acquisition (P. Van Hentenryck)
\end{itemize}
\noindent

\bibliography{references}
\end{document}